\documentclass[12pt]{amsart}

\usepackage[english]{babel}
\usepackage[ansinew]{inputenc}
\usepackage{amsmath, amsthm, amssymb, graphicx, mathrsfs}
\usepackage{color}
\usepackage{hyperref}
\usepackage{verbatim}
\usepackage{pinlabel}

\usepackage[headings]{fullpage}

\usepackage{amsmath,amsthm,amssymb}
\numberwithin{equation}{section}

\newtheorem{thm}{Theorem}[section]
\newtheorem{lem}[thm]{Lemma}
\newtheorem{cor}[thm]{Corollary}

\newtheorem{prop}[thm]{Proposition}
\theoremstyle{definition}
\newtheorem{defn}[thm]{Definition}

\newtheorem{conv}[thm]{Convention}
\newtheorem{rem}[thm]{Remark}
\newtheorem{rem-conv}[thm]{Remark-Convention}
\newtheorem{exmp}[thm]{Example}
\newtheorem{prop-defn}[thm]{Proposition-Definition}
\newtheorem{defn-rem}[thm]{Definition-Remark}

\newtheorem*{claim*}{Claim}
\newtheorem*{ack*}{Acknowledgements}
\newtheorem*{ex*}{Example}

\newcommand{\cal}{\mathcal}

\newcommand{\g}{\mathfrak{g}}

\def\epsilon{\varepsilon}
\def\phi{\varphi}
\def\Pr{\mathbb P}

\newcommand{\Curr}{\mbox{Curr}}
\newcommand{\Out}{\mbox{Out}}
\newcommand{\Aut}{\mbox{Aut}}

\newcommand{\ILT}{\mbox{ILT}}

\newcommand{\FN}{F_N}   

\newcommand{\CVN}{\mbox{CV}_N}
\newcommand{\CVNbar}{\overline{\mbox{CV}}_N}

\newcommand{\PCurr}{\Pr\Curr(\FN)}

\newcommand{\R}{\mathbb R}
\newcommand{\Z}{\mathbb Z}

\newcommand{\N}{\mathbb N}

\newcommand{\A}{{\cal A}}

\def\strutdepth{\dp\strutbox}
\def \ss{\strut\vadjust{\kern-\strutdepth \sss}}
\def \sss{\vtop to \strutdepth{
\baselineskip\strutdepth\vss\llap{$\diamondsuit\;\;$}\null}}

\def\strutdepth{\dp\strutbox}
\def \sst{\strut\vadjust{\kern-\strutdepth \ssss}}
\def \ssss{\vtop to \strutdepth{
\baselineskip\strutdepth\vss\llap{$\spadesuit\;\;$}\null}}

\def\strutdepth{\dp\strutbox}
\def \ssh{\strut\vadjust{\kern-\strutdepth \sssh}}
\def \sssh{\vtop to \strutdepth{
\baselineskip\strutdepth\vss\llap{$\heartsuit\;\;$}\null}}

\def\qed{\hfill\rlap{$\sqcup$}$\sqcap$\par}
\def\bar{\overline}
\def\tilde{\widetilde}

\title[Dynamics of hyperbolic automorphisms of $\FN$]
{
North-South dynamics of hyperbolic free group automorphisms on the 
space of currents
}

 \author{Martin Lustig} 
 \address{\tt 
Aix Marseille Universit\'e, CNRS, Centrale Marseille, I2M UMR 7373
13453, Marseille, France}
 \email{\tt Martin.Lustig@univ-amu.fr} 

\author{Caglar Uyanik}
\address{\tt Department of Mathematics, University of Illinois at
 Urbana-Champaign, 1409 West Green Street, Urbana, IL 61801, USA\\
\url{http://www.math.uiuc.edu/~cuyanik2}} \email{\tt cuyanik2@illinois.edu}

\thanks{\today}

\begin{document}
	
\begin{abstract} Let $\varphi$ be a hyperbolic outer automorphism of a non-abelian free group 
$\FN$ 
such that $\varphi$ and $\varphi^{-1}$ admit 
absolute 
train track representatives. We prove that 
$\varphi$ acts on the space of projectivized geodesic currents 
on $\FN$ 
with generalized uniform 
North-South dynamics.  
\end{abstract}

\thanks{The authors gratefully acknowledge support from U.S. National Science Foundation grants DMS 1107452, 1107263, 1107367 ``RNMS: GEometric structures And Representation varieties" (the GEAR Network)." 
The first author was partially supported by the French research grant ANR-2010-BLAN-116-01 GGAA. The second author was partially supported by the NSF grants of Ilya Kapovich (DMS 1405146 and DMS 1710868) and Christopher J. Leininger (DMS 1510034). 
}


\maketitle

\section{Introduction}
 
Thurston's groundbreaking work on the classification of surface homeomorphisms (Nielsen-Thurston classification) has triggered much of the progress
on the mapping class group and on the Teichm\"uller space in the past 30 years. 
This theory has inspired important 
developments for several related groups; 
most notably for the group $\Out(\FN)$ of outer automorphisms of a non-abelian free group $\FN$ of finite rank $N\ge2$.

Two rather different spaces on which $\Out(\FN)$ acts serve as analogues to Teichm\"uller space:
One is Culler-Vogtmann's {\em Outer space} $\CVN$ \cite{CV}, the other  is the space of \emph{projectivized geodesic currents}  $\PCurr$ \cite{Bo}.  An \emph{intersection form}, generalizing Thurston's celebrated intersection form for measured laminations \cite{Th}, intimately intertwines the two spaces \cite{KL2}. 

The space $\CVN$ is finite dimensional, and the action of $\Out(\FN)$ on $\CVN$ is properly discontinuous. 
Moreover, $\CVN$ has a natural ``Thurston boundary'', which gives rise to a compactification $\CVNbar$. The $\Out(\FN)$-action extends to this compactification,
which mimics 
closely Thurston's compactification of 
the
Teichm\"uller space via projective measured laminations \cite{BF93}, \cite{CL95}, \cite{CV}, \cite{Th}. 

On the other hand, the space of projectivized geodesic currents $\PCurr$ is infinite dimensional, but it is already compact, and has a 
projective structure. There is a natural ``interior'' of $\PCurr$ on which $\Out(\FN)$ acts properly discontinuously 
\cite{KL7}. 

The space 
$\CVN$ has received more attention in recent years than 
$\PCurr$, 
due to the connections with various curve complex analogues (see e.g. \cite{BF14}, \cite{HM13}). Nevertheless, 
for several purposes, including some rather fundamental 
ones, 
the space $\PCurr$ seems to have an advantage over $\CVNbar$. In order get a deeper understanding 
one often needs to utilize the action of $\Out(\FN)$ on both spaces simultaneously, see for instance \cite{BR1}, \cite{H12}.

This paper concerns one of the fundamental features prominently present in Nielsen-Thurston
theory: the dynamics of 
the action of
individual elements $\phi \in \Out(\FN)$ on $\PCurr$, and the study of the relationship between aspects of this dynamics and the algebraic structure of $\phi$. 
More precisely, the paper aims to generalize Thurston's result 
that any pseudo-Anosov mapping class 
acts on compactified Teichm\"uller space with 
North-South dynamics 
\cite{Th, Ivanov} 

An important result in this direction is already known for both of the above mentioned $\Out(\FN)$-analogues of Teichm\"uller space:
If the automorphism 
$\phi\in\Out(\FN)$ is fully irreducible, also known as 
iwip (``irreducible with irreducible powers''),  
then the $\phi$-action 
on both, $\CVNbar$ and $\PCurr$, 
has uniform North-South dynamics, see \cite{LL}, \cite{Martin}, \cite{U2} and \cite{U}. The pairs of poles of these two actions are strongly related to each other through the above mentioned intersection form, see 
\cite{CHL3}, 
\cite{KL3}, \cite{U}. 

In this paper we concentrate on 
the 
following class of automorphisms: 
An element
$\phi \in \Out(\FN)$ is called {\em hyperbolic} 
if the length of any non-trivial conjugacy class grows exponentially under iteration of $\phi$.
This condition turns out to be equivalent to requiring that 
$\phi$ is {\em atoroidal}, i.e. $\phi$ has no non-trivial periodic conjugacy class in $\FN$. 

Recall that both classes, fully irreducible and hyperbolic automorphisms, can be viewed as natural analogues of pseudo-Anosov mapping classes. 
In a probabilistic sense, both of these notions are ``generic''  in $\Out(F_N)$\cite {Ri}, as is the case for pseudo-Anosov homeomorphisms in mapping class groups \cite{Ri,Mah}. 
Hyperbolic automorphisms have the advantage that their mapping torus groups $F_N\rtimes_\varphi\mathbb{Z}$ are Gromov hyperbolic, see \cite{BF}, \cite{Brink}.
However, hyperbolic automorphisms can admit invariant free factors and hence these automorphisms are harder to study than fully irreducible automorphisms. 
As a consequence of the more complicated algebraic structure, one can not expect a classical North-South dynamics for hyperbolic automorphisms. Counterexamples for both, $\CVNbar$ and $\PCurr$ are easy to construct.

Nevertheless, 
it turns out that there are still strong North-South dynamical features 
in the 
action 
of any hyperbolic $\phi \in \Out(\FN)$ 
on 
the space of currents. 
This paper aims 
to make this 
statement 
precise:
We define ``generalized North and South poles'' in $\PCurr$, 
which are finite dimensional simplexes
$\Delta_+(\phi)$ and $\Delta_-(\phi)$, 
and show:

\begin{thm}
\label{mainthm-intro} 
Let $\varphi\in\Out(F_N)$ be a hyperbolic outer automorphism with the property that both $\varphi$ and $\varphi^{-1}$ admit (absolute) train track representatives. Then $\varphi$ acts on $\mathbb{P}\Curr(F_N)$ with ``generalized uniform North-South dynamics 
from $\Delta_-(\phi)$ to $\Delta_+(\phi)$'' in the following sense: 

Given a neighborhood $U$ of $\Delta_{+}(\varphi)$ and a compact set $K\subset\mathbb{P}\Curr(F_N)\smallsetminus\Delta_{-}(\varphi)$, there exists an integer $M \geq 1$ such that $\varphi^{n}(K)\subset U$ for all $n\ge M$. 

Similarly, given a 
neighborhood $V$ of $\Delta_{-}(\varphi)$ and  a compact set $K'\subset\mathbb{P}\Curr(F_N)\smallsetminus\Delta_{+}(\varphi)$, there exists an integer $M' \geq 1$ such that $\varphi^{-n}(K')\subset V$ for all $n\ge M'$. 
\end{thm}

Although there exist hyperbolic automorphisms that do not admit (absolute) train track representatives (see Sections 
\ref{graphmaps}
and \ref{markings}), the train track assumption in the above theorem turns out to be not really restrictive
(see Remark \ref{non-ht}). 
Our train track technology (see Section \ref{train-tracks}) has been purposefully set up in a way that it also applies to more general kind of train track maps. 
Since the paper is already technically rather loaded, we have refrained from adding such generalizations here.

\smallskip

The second dynamics result obtained here concerns single orbits of \emph{rational currents}, that is, currents given by non-trivial conjugacy classes in $\FN$, which form a dense subset of $\mathbb{P}\Curr(F_N)$, see Section \ref{sec:currents}:

\begin{thm}\label{rationalconv}
Let $\varphi$ be as in Theorem \ref{mainthm-intro}, and let $\mu\in \Curr(F_N)$
be 
a non-zero rational current.
 
Then, after possibly replacing $\phi$ by a positive power,
there are 
$[\mu_{\infty
}]\in \Delta_{+}(\varphi)$ and $[\mu_{-\infty}]\in\Delta_{-}(\varphi)$ such that 
\[
\lim_{n\to\infty}\varphi^{n}[\mu]=[\mu_{\infty}]
\]
and 
\[
\lim_{n\to\infty}\varphi^{-n}[\mu]=[\mu_{-\infty}]. 
\]
\end{thm}

It is important to point out that this pointwise convergence is in general not uniform: On the contrary, it is possible to construct examples of hyperbolic $\phi$ for which there exist rational currents that linger for an arbitrary long iteration time in an arbitrary close neighborhood of a (fixed) current that is different from the eventual limit current.

We now give a brief outline of the 
paper. We'd like to point out immediately that the sections 3, 4 and 5 are written independently of each other and can be read in arbitrary order.

\begin{itemize}
\item
In Section \ref{prelim} we recall some background information and set up notation that is 
needed in the sections to follow.

\item
In Section \ref{NSD} we define several versions of North-South dynamics (in a  general context - without reference to $\Out(\FN)$ or currents) and derive the main criterion 
(Propositions \ref{convergence-criterion} and \ref{NS-for-roots})
used in the proof of Theorem \ref{mainthm-intro}. 
\item
Section \ref{train-tracks} is entirely devoted to train track technology, addressing issues such as indivisible Nielsen paths (INPs), exponential growth of the number of illegal turns under backwards iteration, etc. Moreover, we recall and investigate the notion of \emph{``goodness"} which plays a key role in our analysis. We conclude Section \ref{train-tracks} with Proposition \ref{backforthgoodness} 
(see also Remark \ref{back-to-ht}). This proposition is a crucial tool in our approach, since  it permits the restriction of our attention in the convergence approximation to conjugacy classes with large goodness.
\item
In Section \ref{symbolic} 
we present some standard notions from symbolic dynamics and use the main result from our previous paper \cite{LU1} on limit frequencies of substitutions in order to 
define the limit 
simplexes 
$\Delta_+(\phi)$ and $\Delta_-(\phi)$
and show their $\phi$-invariance (Corollary \ref{delta-invariance}).
\item
In Section \ref{convergence} the crucial convergence approximation is calculated for conjugacy classes with large goodness (Lemma \ref{rationalconvergence}), done carefully over several pages. The main results are derived from ``putting together the pieces'' derived previously.
We finish this section with a particular 
``free products'' 
example of 
hyperbolic outer automorphisms which are not fully irreducible. 
\end{itemize}
{\bf Acknowledgements:}  We would heartfully like to thank Ilya Kapovich for his constant encouragement and support, as well as for 
many helpful mathematical and organizational suggestions.
We would also like to 
thank Arnaud Hilion and Matt Clay for useful 
exchanges and 
helpful comments. 
We thank the anonymous referee who provided detailed feedback, and prompted us to clarify several points and improve the exposition. 
The second author is grateful to Chris Leininger for useful discussions, support, and helpful comments.
\section{Preliminaries}
\label{prelim}

We first recall facts about graphs, paths, and maps between graphs. 
For an even more detailed discussion of these notions we refer reader to Section 2 of \cite{DKL}. 
For a streamlined version which uses the standard conventions the reader is referred to \cite{LUv2}.

\subsection{Graphs}
\label{graphs}
A \emph{topological graph} $\Gamma$ is a 1-dimensional finite cell
complex. The $0$-cells of $\Gamma$ are called \emph{vertices} and the set of vertices is denoted by $V\Gamma$.  The $1$-cells of $\Gamma$ are called \emph{topological edges} and the set of topological edges is denoted by $E_{top}\Gamma$. Each topological edge admits exactly two orientations, and we call a topological edge with a choice of orientation an \emph{oriented edge}. The set of all oriented edges is denoted by $E\Gamma$. Given $e\in E\Gamma$, we denote the edge with the opposite orientation by $\bar{e}$. Given a finite set of points in a topological graph $\Gamma$, one obtains a topological graph $\Gamma'$ as the \emph{subdivision}, in which the vertex set contains the original vertex set, together with the finite set of points.

Let $\Gamma$ be a topological graph. We choose an orientation for each edge, and denote the set of positively (resp. negatively) oriented edges by $E^{+}\Gamma$ ( resp. $E^{-}\Gamma$). Given an edge $e\in E^{+}\Gamma$, the initial vertex of $e$ is denoted by $o(e)$ and the terminal vertex of $e$ is denoted by  $t(e)$.

The graph $\Gamma$ is equipped with a natural metric called the \emph{simplicial metric} which is obtained by identifying each edge $e$ of $\Gamma$ with the interval $[0,1]$ via a characteristic map $\alpha_{e}:[0,1]\to e$, meaning that $\alpha$ is continuous on $[0,1]$ and the restriction to $(0,1)$ is an orientation preserving homeomorphism onto the interior of $e$.   

\subsection{Paths}
A \emph{combinatorial edge path} $\gamma$ of \emph{simplicial length $n\ge1$} is a concatenation $e_1e_2\ldots e_n$ of oriented edges of $\Gamma$ where $t(e_i)=o(e_{i+1})$.  The \emph{simplicial length} of an edge path $\gamma$ in $\Gamma$ is denoted by $|\gamma|_{\Gamma}$, and if it is clear from the context, we  suppress $\Gamma$ and write $|\gamma|$.

A \emph{topological edge path of simplicial length $n\ge1$} is a continuous map $p:[a,b]\to\Gamma$ such that there exists a combinatorial edge path $\gamma_p=e_1e_2\ldots e_n$ and a subdivision $a=c_0<c_1<c_2<\ldots<c_n=b$ such that for each $1\le i\le n$ we have $p(c_{i-1})=o(e_i)$, and $p(c_i)=t(e_i)$ and the restriction of $p$ to $(c_{i-1},c_i)$ is an orientation preserving homeomorphism onto the interior of $e_i$. We say that $\gamma_p$ is the combinatorial edge path associated to $p$. 

Let $\Gamma$ be equipped with the simplicial metric as above. A topological edge path $p:[a,b]\to\Gamma$ is a \emph{PL edge path} if for every edge $e_i$ in the combinatorial edge path $\gamma_p$ the composition $ \alpha_{e_i}^{-1}\circ p:(c_{i-1},c_{i})\to(0,1)$ is the unique orientation preserving affine homeomorphism from $(c_{i-1},c_i)$ to $(0,1)$. By not requiring that $p(a)$ and $p(b)$ are vertices in $\Gamma$, the concept of PL edge paths easily extend to \emph{PL paths}.

A \emph{combinatorial edge path of simplicial length $0$} consists of a vertex in $\Gamma$. A (degenerate) edge path $\gamma$ that consists of a vertex is called \emph{trivial}. A combinatorial edge path $\gamma$ is called \emph{reduced} if $e_i\neq\bar{e}_{i+1}$ for all $i=1,\ldots, n-1$. A reduced edge path is called \emph{cyclically reduced} if $t(e_n)=o(e_1)$ and $e_n\neq \bar{e}_1$.  A topological or PL  edge path is called \emph{reduced (resp. trivial)} if the corresponding combinatorial edge path is reduced (resp. trivial). More generally, a PL path is called \emph{reduced} if $p:[a,b]\to\Gamma$ is locally injective. 

\subsection{Graph maps}\label{graphmaps}
A \emph{topological graph map} $f:\Gamma\to\Gamma$ is a continuous map that sends vertices to vertices, and edges to topological edge paths. The latter means that for any edge $e$ and a characteristic map $\alpha_e:[0,1]\to \Gamma$,
the composition $f\circ\alpha_{e}:[0,1]\to\Gamma$ is a topological edge path as above. 

A \emph{linear graph map} is a topological graph map $f:\Gamma\to\Gamma$ as above which, in addition, satisfies the following condition: For every $e\in\Gamma$, the restriction of $f$ to $e$ is a PL edge path from $f(o(e))$ to $f(t(e))$. 

Note that given a topological graph map $f:\Gamma\to\Gamma$  there is a linear graph map from $\Gamma$ to itself that is homotopic to $f$ relative to vertices.

\begin{defn}
\label{expanding-and-D}
We say that a topological graph map $f:\Gamma\to\Gamma$ is \emph{regular} if for all $e \in E\Gamma$ the (combinatorial) edge path (associated to) $f(e)$ is non-trivial and reduced.  The map $f$ is called {\em expanding} if for every edge $e\in E\Gamma$ there is an exponent $t \geq 1$ such that the edge path $f^t(e)$ has simplicial length $|f^t(e)| \geq 2$.
\end{defn}

A turn in $\Gamma$ is a pair $(e_1,e_2)$ 
where $o(e_1)=o(e_2)$. A turn is called \emph{non-degenerate} if $\bar{e}_1\neq e_2$, otherwise it is called  \emph{degenerate}. A regular graph map $f:\Gamma\to\Gamma$ induces a \emph{derivative map} $Df:E\Gamma\to E\Gamma$ where $Df(e)$ is the first edge of the combinatorial edge path associated to $f(e)$. The derivative map induces a map $Tf$ on the set of turns defined as $Tf((e_1,e_2)):=(Df(e_1),Df(e_2))$.  A turn 
$(e_1,e_2)$ 
is called \emph{legal} if $Tf^{n}((e_1,e_2))$ is non-degenerate for all $n\ge0$. An edge path $\gamma=e_1e_2\ldots e_n$ is called legal if every turn $(e_i,\bar{e}_{i+1})$ in $\gamma$ is legal. 

\begin{rem-conv}
(1) 
For the rest of this paper, a \emph{path} means a \emph{PL path}. Also note that, given a combinatorial edge path $\gamma$ in $\Gamma$ there is always a PL edge path $p:[a,b]\to\Gamma$ such that $\gamma_p=\gamma$. Hence, by an edge path we will either mean a PL edge path or a combinatorial edge path and won't make a distinction unless it is not clear from the context. 

\noindent
\smallskip
(2)
Similarly, from now on,  unless otherwise stated, a {\em graph map} means a linear graph map as defined above.
\end{rem-conv}

\begin{defn}
A graph map $f:\Gamma\to\Gamma$ is called a \emph{train track map} if 
for every edge $e$ and any $t \geq 1$ the edge path $f^{t}(e)$ is legal.
\end{defn}

\subsection {Markings and topological representatives}\label{markings}

Let $\FN$ denote a free group of finite rank $N \geq 2$ which we fix once and for all.
A graph $\Gamma$ is said to be {\em marked}, if it is equipped with a {\em marking isomorphism} $\theta: \FN \overset{\cong}{\longrightarrow} \pi_1 \Gamma$. Markings are well defined only up to inner automorphisms, which is why we suppress the choice of a base point in $\Gamma$.

A homotopy equivalence $f: \Gamma \to \Gamma$ is a {\em topological representative} of an automorphism $\phi \in \Out(\FN)$ with respect to the marking $\theta: \FN \overset{\cong}{\longrightarrow} \pi_1 \Gamma$ if one has
$\phi = \theta^{-1} f_* \theta$.
A graph map $f:\Gamma\to\Gamma$ 
is called a \emph{train track representative} for $\varphi$ if $f$ is a topological representative for $\varphi$ and $f$ 
is a train track map. 

\begin{rem}
\label{make-expanding}
If a topological graph map $f: \Gamma \to \Gamma$ represents a hyperbolic outer automorphism 
$\phi$ of $F_N$, then the hypothesis that $f$ be expanding is easy to satisfy:  It suffices to contract all edges which are not expanded by any iterate $f^t$ to an edge path of length $\geq 2$:  The  contracted subgraph must be a forest, as otherwise some $f^t$ would fix a non-contractible loop and hence $\phi^t$ would fix a non-trivial conjugacy class of $\pi_1(\Gamma)\cong F_N$,
contradicting the assumption that $\phi$ is hyperbolic. It is easy to see that, if $f$ is a train track map, then this property is inherited by the above ``quotient map''.
\end{rem}

Given an (not necessarily reduced) edge path $\gamma\in\Gamma$, let $[\gamma]$ denote the reduced edge path which is homotopic to $\gamma$ relative to endpoints. The following is a classical fact for free group automorphisms:

\begin{lem}[Bounded Cancellation Lemma]\cite{Coo}\label{BCL} Let $f:\Gamma\to\Gamma$ be topological graph map that is a homotopy equivalence. There exist a constant $C_{f}$, depending only on $f$, such that for any reduced 
edge path $\rho=\rho_{1}\rho_{2}$ in $\Gamma$ one has
\[
|[f(\rho)]|\ge|[f(\rho_1)]|+|[f(\rho_2)]|-2C_{f}.
\]
That is, at most $C_f$ terminal edges of $[f(\rho_1)]$ 
are cancelled against 
$C_f$ initial edges of $[f(\rho_2)]$ when we concatenate them to obtain $[f(\rho)]$. 
\end{lem}

\begin{rem}
\label{non-ht}
In 
Section \ref{train-tracks} of
this paper we will consider train track maps $f: \Gamma \to \Gamma$ that are not assumed to be homotopy equivalences. We do, however, assume that $f$ possesses a {\em cancellation bound}: there is an integer $C \geq 0$ such that for any two legal edge paths $\gamma$ and $\gamma'$ with common initial vertex but distinct initial edges the backtracking path $\gamma''$ at the image of the concatenation point of the non-reduced path $f(\bar \gamma \circ \gamma')$ has combinatorial length $|\gamma''| \leq C$.  The smallest such integer $C$ is denoted by $C(f)$.

The reason for not imposing the homotopy equivalence condition is 
to make the results and techniques of this paper available for the study of free group endomorphisms, as well as for the use of more general type of train track maps in the context of free group automorphisms.
\end{rem}

\begin{defn}
\label{defn-INPs}
A path $\eta$ in $\Gamma$ which crosses over precisely one illegal turn is called a {\em periodic indivisible Nielsen path} (or {\em INP}, for short), if for some exponent $t \geq 1$ one has $[f^t(\eta)] = \eta$. The smallest such $t$ is called the period of $\eta$. A path $\gamma$ is called \emph{pre-INP} if its image under $f^{t_0}$ is an INP for some $t_0\ge1$. 
The illegal turn on $\eta = \gamma' \circ  \bar{\gamma}$ is called the {\em tip} of $\eta$, while the two maximal initial legal subpaths $\gamma'$ and $\gamma$, of $\eta$ and $ \bar{\eta}$ respectively, are 
called the {\em branches} of $\eta$. 
A \emph{multi-INP} or a \emph{Nielsen path} is a legal concatenation of finitely many INP's. 
\end{defn}
 
If $f$ is an expanding train track map, then any non-trivial reduced path $\eta$, such that $f^t(\eta)$ is homotopic to $\eta$ relative to endpoints for some $t\ge 1$, can be written as a legal concatenation of finitely many INP's.
This can be seen by a direct elementary argument; alternatively it follows from 
Proposition 
\ref{pseudo-legal-iterate}. 

The reader should be aware of the fact that a priori the endpoints of an INP may not be vertices of $\Gamma$ (compare with Convention \ref{endpoints-of-INPs} below).

\subsection{Geodesic Currents on Free Groups}
\label{sec:currents}
Let $\FN$ be the free group of finite rank $N \geq 2$, and
denote by $\partial F_N$ the Gromov boundary of $F_N$. The double boundary of $F_N$ is defined by 

\[
\partial^{2}F_N:=(\partial F_N\times\partial F_N)\smallsetminus\Delta,
\]
where $\Delta$ denotes the diagonal.
The {\em flip map} on $\partial^2\FN$ is defined by
\[
(\xi,\zeta) \mapsto (\zeta, \xi).
\]
A \emph{geodesic current} 
$\mu$
on $\FN$ is a 
locally finite, Borel measure on $\partial^2\FN$ which is flip and $F_N$-invariant. The set of all geodesic currents on $\FN$ is denoted by $\Curr(F_N)$. The space of currents is naturally equipped with the 
weak*-topology. In particular, a sequence of currents $\mu_{n}\in \Curr(\FN)$ satisfies
\[
\lim_{n\to\infty}\mu_{n}=\mu\ \  \text{iff}\ \   \lim_{n\to\infty}\mu_n(S_1\times S_2)=\mu(S_1\times S_2)
\]
for any two disjoint Borel subsets $S_1,S_2\subset \partial\FN$.

An immediate example of a geodesic current is 
given by any non-trivial
conjugacy class in $\FN$: For an element $g\in\FN$ which is not a proper power (i.e. $g\ne h^{k}$ for 
any $h\in\FN$ and 
any $k\ge2$)  define the \emph{counting current} $\eta_g$ 
as follows:
for any Borel set $S\subset\partial^{2}\FN$
the value 
$\eta_g(S)$ is the number of $\FN$-translates of $(g^{-\infty},g^{\infty})$ or of  $(g^{\infty},g^{-\infty})$ 
that 
are contained in 
$S$. For an arbitrary element 
$h\in\FN \smallsetminus \{1\}$ 
write $h=g^{k}$, where $g$ is not a proper power, and define 
$\eta_{h}:=k\eta_{g}$. It is easy to see that this definition is independent of the particular representative of the conjugacy class of 
$h$, 
so that $\eta_h$ depends only on $[h] \subset \FN$.
The set of all non-negative real multiples of a counting current is called \emph{rational currents}, and it is an important fact that rational currents are dense in $\Curr(\FN)$, see \cite{Ka1}, \cite{Ka2}. 

Let 
$\Phi\in \Aut(F_N)$ 
be an automorphism of $\FN$.  Since $\Phi$ 
induces
homeomorphisms on $\partial \FN$ and on $\partial^2\FN$,
there is a natural action of $Aut(F_N)$ on the space of currents:
For any $\Phi\in Aut(\FN)$, $\mu\in \Curr(\FN)$ and 
any Borel subset $S \subset \partial^{2}F_N$ 
one has
\[
\Phi\mu(S):=\mu(\Phi^{-1}(S)). 
\]
Note that any inner automorphism of $\FN$ acts trivially, so that the above action factors through the quotient 
of $\Aut(\FN)$ by all inner automorphisms to give a well defined action of $\Out(F_N)$ on $\Curr(\FN)$. 

The space of \emph{projectivized geodesic currents} is defined as 
the quotient 
\[
\mathbb{P}\Curr(\FN)=:\Curr(F_N)\big/\sim
\]
where $\mu_1\sim\mu_2$ in $\Curr(F_N)$ if and only $\mu_1=\lambda\mu_2$ for some $\lambda>0$. We denote the projective class of the current $\mu$ by $[\mu]$. 

We derive from the above stated facts:

\begin{prop}
\label{current-space}
(1)
The space $\mathbb{P}\Curr(F_N)$ is a compact 
metric space, and the above actions of $Aut(F_N)$ and $\Out(F_N)$ on $\Curr(F_N)$ descend to well defined actions on $\mathbb{P}\Curr(F_N)$ by 
setting
\[
\phi[\mu]:=[\phi\mu]. 
\]

\smallskip
\noindent
(2)
The subset of $\PCurr$ defined by all rational currents is dense in $\PCurr$.
\end{prop}

For any marked graph $\Gamma$ the marking isomorphism 
$\theta: \FN \to \pi_1 \Gamma$
gives rise to canonical $\FN$-equivariant homeomorphisms between the Gromov boundaries and double boundaries of the group $F_N$ and of the universal cover  $\tilde{\Gamma}$, which for simplicity we also denote by $\theta$.  

Let $\tilde{\gamma}$
be a reduced edge path in $\tilde\Gamma$. We define the \emph{cylinder set} associated to $\tilde{\gamma}$ as follows:
\[
Cyl(\tilde{\gamma}):=\{(\xi,\zeta)\in\partial\FN\mid\tilde{\gamma}\ \text{is a subpath of}\ [\partial\theta(\xi),\partial\theta(\zeta)]\}
\]
where
$[\partial\theta(\xi),\partial\theta(\zeta)]$ denotes the biinfinite reduced edge path that joins 
the end $\partial\theta(\xi)$ to the end 
$\partial\theta(\zeta)$ of $\tilde{\Gamma}$. 

Let $\gamma$
be a reduced edge path in $\Gamma$ and let $\mu\in \Curr(F_N)$. Let $\tilde{\gamma}$ be a lift of $\gamma$ to $\tilde{\Gamma}$. 
We define 
\[
\langle\gamma,\mu\rangle
:= \mu(Cyl(\tilde{\gamma}))
\]
and note 
that the $F_N$-equivariance for geodesic currents 
implies 
that the value $\mu(Cyl_{\alpha}(\tilde{\gamma}))$ does not depend on the particular lift 
$\tilde \gamma$ 
we chose, so we have a well defined quantity $\langle\gamma,\mu\rangle$. 

Note also
that in case of a rational current $\mu = \eta_g$ 
the quantity $\langle\gamma,\mu\rangle
$ is given by the number of 
{occurrences} of $\gamma$ 
or of $\bar{\gamma}$ in 
the reduced loop 
$[\gamma]$ 
which represents the conjugacy class $[g]$.
Moreover, the action of $\Out(F_N)$ on rational currents is given explicitly by the formula
\begin{equation}
\label{rational-image}
\varphi \eta_g=\eta_{\varphi(g)}.
\end{equation}

An important fact about 
cylinder sets is that 
they 
form a basis for the topology on $\partial^{2}\FN$, 
so that 
a geodesic current $\mu$ is uniquely determined by the set of values $\langle\gamma,\mu\rangle$ 
for all non-trivial reduced edge paths $\gamma$ in $\Gamma$, see \cite{Ka2}.  

In particular, 
we have $\underset{n\to\infty}{\lim}\mu_{n}=\mu$ if and only if $\underset{n\to\infty}{\lim}\langle\gamma,\mu_{n}\rangle=\langle\gamma,\mu\rangle$ for all reduced edge paths $\gamma$ in $\Gamma$. 

Moreover, if we denote by $\cal P(\Gamma)$ the set of reduced edge paths in $\Gamma$, then the function 
$$\mu_\Gamma: \cal P(\Gamma) \to \R_{\geq 0}, \, \gamma \mapsto \langle\gamma,\mu\rangle$$
satisfies for all $\gamma \in \cal P(\Gamma)$:
\begin{enumerate}
\item
$\mu_\Gamma(\gamma) = \underset{\gamma' \in E_+(\gamma)}{\sum} \mu_\Gamma(\gamma') = \underset{\gamma' \in E_-(\gamma)}{\sum} \mu_\Gamma(\gamma'')$, and
\item
$\mu_\Gamma(\gamma) = \mu_\Gamma(\bar\gamma)$,
\end{enumerate}
where $E_+(\gamma) \subset \cal P(\Gamma)$ and  $E_-(\gamma) \subset \cal P(\Gamma)$ denote the set of reduced edge paths obtained from $\gamma$ by adding a further edge at the beginning or at the end respectively.

Any function $\mu'_\Gamma: \cal P(\Gamma) \to \R_{\geq 0}$ which satisfies (1) and (2) above is called a {\em Kolmogorov function}, and it is well known (see \cite{Ka1, Ka2}) that any such $\mu'_\Gamma$ defines uniquely a geodesic current $\mu$ on $\FN$ which satisfies $\mu_\Gamma = \mu'_\Gamma$.

Let 
$\Gamma$ be a marked graph
as above. Given a geodesic current $\mu\in \Curr(F_N)$, define the weight of $\mu$ with respect to $\Gamma$ as
\begin{equation}
\label{current-norm}
\|\mu\|_\Gamma:=\sum_{e\in E^+\Gamma}\langle e,\mu\rangle.
\end{equation}
The following criterion given in \cite{Ka2} plays a key role in our convergence estimates in 
Section 
\ref{convergence}. 

\begin{lem} Let $([\mu_n])_{n \in \N}$ 
be a sequence of currents in $\mathbb{P}\Curr(F_N)$. Then 
one has 
$\underset{n\to\infty}{\lim}[\mu_n]=[\mu]$ in $\mathbb{P}\Curr(F_N)$ if and only if 
\[
\lim_{n\to\infty}\frac{\langle \gamma,\mu_n\rangle}{\|\mu_n\|_{\Gamma}}=\frac{\langle \gamma,\mu\rangle}{\|\mu\|_{\Gamma}}
\]
holds 
for all reduced edge paths $\gamma$ in $\Gamma$. 
\end{lem}

\section{North-South dynamics}
\label{NSD}
In this section we describe some general considerations for 
maps with North-South dynamics. We will keep the notation simple and general; at no point we will refer to the specifics of currents on free groups.  However, in this section we will prove the main criteria used in the remainder of the paper to establish the North-South dynamics result in our main theorem.

\begin{defn}
\label{NS-dyn} 
Let 
$f: X \to X$ be homeomorphism 
of a 
topological 
space $X$. 

(a) The map $f$ is said to have {\em(pointwise) North-South dynamics} if $f$ has two distinct fixed points $P_{+}$ and $P_{-}$, called \emph{attractor} and \emph{repeller}, such that for every $x \in X \smallsetminus \{P_+, P_-\}$ one has:
\[
\lim_{t \to \infty} f^t(x) = P_+ \qquad {\rm and} \qquad  \lim_{t \to - \infty} f^t(x) = P_{-}.
\]

\smallskip

(b)
The 
map $f$ is said to have {\em uniform North-South dynamics} if the following holds:
There exist two distinct
fixed points $P_{-}$ and $P_{+}$ of $f$, such that for every compact set $K \subset X \smallsetminus \{P_-\}$ 
and every neighborhood $U_+$ of $P_+$ there exists an integer $t_+ \geq 0$ such that for every $t \geq t_+$ one has:
\[
f^t(K) \subset U_{+}.
\]
Similarly, for every compact set $K \subset X \smallsetminus \{P_+\}$ 
and every neighborhood $U_-$ of $P_-$ there exists an integer $t_- \leq 0$ such that for every $t \leq t_-$ one has:
\[
f^t(K) \subset U_-.
\]
\end{defn}

It is easy to see that uniform North-South dynamics implies pointwise North-South dynamics. 

\begin{defn}
\label{generalized-NS-dyn}
A homeomorphism $f: X \to X$ of a 
topological 
space $X$
has {\em generalized uniform North-South dynamics} if there exist two disjoint compact $f$-invariant sets $\Delta _+ $ and $\Delta_-$ in $X$, such that the following hold:
\begin{enumerate}
\item[(i)]
For every compact set $K \subset X \smallsetminus \Delta_-$ 
and every neighborhood $U_+$ of $\Delta_+$ there exists an integer $t_+ \geq 0$ such that for every $t \geq t_+$ one has:
\[
f^t(K) \subset U_+
\]
\item[(ii)]
For every compact set $K \subset X \smallsetminus \Delta_+$ 
and every neighborhood $U_-$ of $\Delta_-$ there exists an integer $t_- \leq 0$ such that for every $t \leq t_-$ one has:
\[
f^t(K) \subset U_-
\]
\end{enumerate}
\end{defn}

More precisely, we say that the map $f$
has 
{\em generalized uniform 
North-South dynamics from $\Delta_-$ to $\Delta_+$}. 
Note that we interpret the phrase 
``$f$-invariant'' in its strong meaning, i.e. $f(\Delta_+) = \Delta_+$ and $f(\Delta_-) = \Delta_-$.

\begin{prop}
\label{convergence-criterion}

Let $f:X\to X$ be a homeomorphism of a compact 
metrizable 
space $X$. 
Let 
$Y \subset X$ be dense subset of $X$, and let $\Delta_{+}$ and $\Delta_{-}$ be two $f$-invariant sets in $X$ that are disjoint. Assume that the following criterion holds:

For every neighborhood $U$ of $\Delta_+$ and every neighborhood $V$ of $\Delta_-$ there exists an integer $m_0 \geq 1$ such that for any $m \geq m_0$ and any $y \in Y$ one has either $f^m(y) \in U$ or $f^{-m}(y) \in V$. 

Then $f^{2}$ has generalized uniform North-South dynamics from $\Delta_{-}$ to $\Delta_+$.
\end{prop}

\begin{proof}

Let 
$K \subset X\smallsetminus\Delta_{-}$ be compact, and let $U$ and $V$ be neighborhoods of $\Delta_{+}$ and $\Delta_{-}$ respectively. See Figure \ref{NSDfigure}.

\begin{figure}[h!]
\labellist
\small\hair 2pt
\pinlabel {$\Delta_{-}$} [ ] at 195 570
\pinlabel {$\Delta_{+}$} [ ] at 410 570
\pinlabel {$K$} [ ] at 300 490
\pinlabel {$W$} [ ] at 345 520
\pinlabel {$V$} [ ] at 220 640
\pinlabel {$U$} [ ] at 400 640
\pinlabel {$U'$} [ ] at 370 590
\pinlabel {\large{$X$}} [ ] at 160 670
\endlabellist
\begin{center}
\includegraphics[scale=0.7]{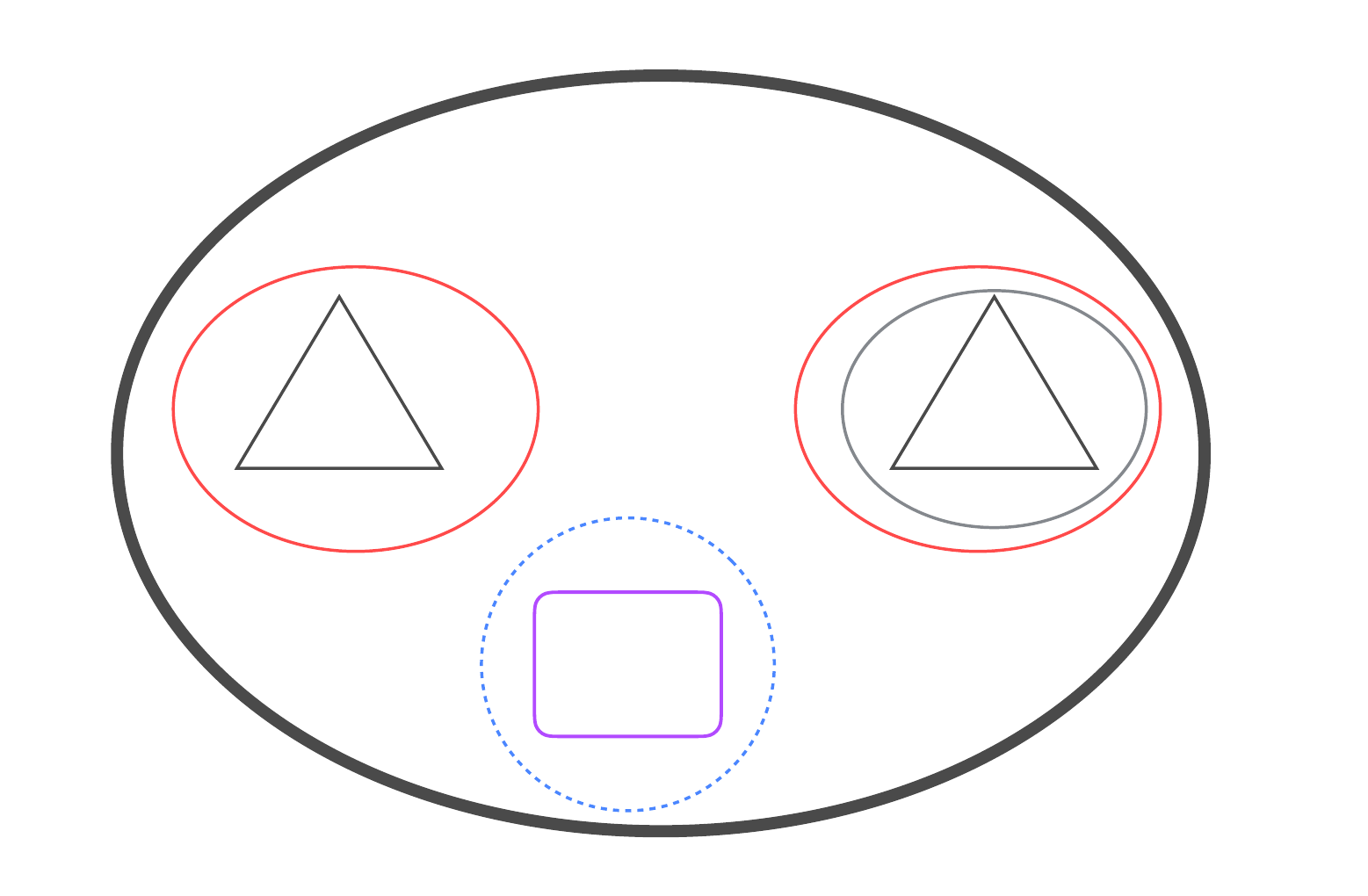}
\caption{}
\label{NSDfigure}
\end{center}
\end{figure}

Let $W$ be an
open neighborhood of $K$ such that $W\cap V=\emptyset$. 
Since $X$ is compact, the closure $\bar{W}$ is compact. 
Then $V_1
:= V \smallsetminus \bar{W}$ is 
again 
an open neighborhood of $\Delta_-$, moreover it is disjoint from $\bar{W}$.  Let $U_1$ be a neighborhood of $\Delta_+$ which has the property that its closure is contained in the interior of $U$.  Such a neighborhood exists because $X$ is a 
metrizable 
space. 

Let $m_0$ be as postulated in the criterion, applied to the neighborhoods $U_1$ and $V_1$, and pick any $m \geq m_0$.  Consider any $y \in Y \cap f^{m}(\bar{W})$.  Notice that $f^{-m}(y)$ is contained in $\bar{W}$, which is disjoint from $V_1$. Thus, by the assumed criterion, $f^m(y)$ must be contained in $U_1$.  

Since $W$ is open and $f$ is a homeomorphism, 
any dense subset of $X$ must intersect $f^{m}(\bar{W})$ in a subset that is dense in $f^{m}(\bar{W})$. This implies that $f^{m}(f^{m}(\bar{W}))\subseteq \bar{U_1}\subset U$. Since $K\subset\bar{W}$, this shows that $f^{2m}(K)\subset U$. 

Using the analogous 
argument for the inverse iteration we see that $f^{2}$ has generalized uniform North-South dynamics from $\Delta_{-}$ to $\Delta_{+}$. 
\end{proof}

\begin{prop}
\label{NS-for-roots}
Let $f: X \to X$ be a homeomorphism of a compact 
space $X$, and let $\Delta_{+}$ and $\Delta_{-}$
be disjoint $f$-invariant sets. Assume that some power $f^s$ with $s \geq 1$ has generalized uniform North-South dynamics from $\Delta_-$ to $\Delta_+$. 

Then the map $f$, too, has generalized uniform North-South dynamics from $\Delta_-$ to $\Delta_+$.
\end{prop}

\begin{proof}
Let $K\subset X\smallsetminus\Delta_{-}$
be compact, 
and let $U$ be an open neighborhood of $\Delta_{+}$.

Set $K' := K \cup f(K) \cup \ldots \cup f^{s-1}(K)$, which is again 
compact. Note that the fact that  $K\subset X\smallsetminus\Delta_{-}$ and $f^{-1}(\Delta_-) = \Delta_-$ implies that $K'\subset X\smallsetminus\Delta_{-}$.  Indeed, $x\in K'$ implies that $x=f^{t}(y)$ for some $y\in K$ and for some $0\le t\le s-1$. 

Thus 
$x\in\Delta_-$ would imply that $y=f^{-t}(x)\in f^{-1}(\Delta_{-})=\Delta_{-}$, contradicting 
the assumption 
$K\cap\Delta_{-}=\emptyset$. 

From the hypothesis that $f^s$ has generalized uniform North-South dynamics from $\Delta_-$ to $\Delta_+$ it follows that there is a bound $t_0$ such that for all $t' \geq t_0$ one has $f^{t' s}(K') \subset U$. 

Hence, for any point 
$x\in K$ and any integer $t \geq s t_0$, 
we can 
write $t = k + s t'$ 
with 
$t' \geq t_0$ and $0 \leq k \leq s-1$ 
to obtain the desired fact
\[
f^{t}(x)=f^{k+st'}(x)=f^{st'}f^{k}(x)\in f^{st}(K')\subset U. 
\]
The analogous 
argument for $f^{-1}$ finishes the proof of the Lemma. 
\end{proof}

\section{Train-Tracks}	
\label{train-tracks}

In 
this section we will consider train track maps which satisfy the following properties:

\begin{conv} 
\label{tt-map-convention}
Let $\Gamma$ be 
a finite connected graph, and 
let $f: \Gamma \to \Gamma$ be 
an expanding train track map which possesses a cancellation bound $C_f \geq 0$ 
as defined in Remark \ref{non-ht}.
We also assume that $f$ has been replaced by a positive power so that for some integers $\lambda'' \geq \lambda' >1$ we have, for any edge $e$ of $\Gamma$:
\begin{equation}
\label{1.1}
\lambda'' \geq |f(e)| \geq \lambda'
\end{equation}
and $\lambda',\lambda''$ is attained for some edges. 
\end{conv}

\subsection{Goodness}

${}^{}$
\smallskip

The following terminology was introduced by R. Martin in his thesis \cite{Martin}. 

\begin{defn} 
\label{goodness}
Let $f: \Gamma \to \Gamma$, $C_f$ and $\lambda'$
be as in Convention \ref{tt-map-convention}. Define the \emph{critical constant} $C$ for $f$ as 
$C:=\dfrac{C_{f}}{\lambda'-1}$. 
Let $\gamma$ be a reduced edge path in $\Gamma$. Any edge $e$ in $\gamma$ that is at least 
$C$ edges away from an illegal turn on $\gamma$ is called \emph{good}, where the distance (= number of edges traversed) is measured on $\gamma$.  An edge is called \emph{bad} if it is not {good}.  
Edge paths or loops, in particular subpaths of a given edge path, which consist entirely of good (or entirely of bad) edges are themselves called {\em good} (or {\em bad}).

For any edge path or a loop $\gamma$ in $\Gamma$
we define the {\em goodness} of $\gamma$ as the following quotient:
$$
\mathfrak{g}(\gamma) := \frac{\#\{\text{good edges of}\,  \gamma\}}{|\gamma|}\in[0,1]
$$

\end{defn}

We will now discuss some basic properties of the goodness of paths and loops. We first consider any legal edge path $\gamma$ in $\Gamma$ of length $|\gamma| = C$ and
compute:
\begin{equation}
\label{elementary}
|f(\gamma)| \geq \lambda' |\gamma|=\lambda' C = C_f + C
\end{equation}

\begin{lem}
\label{growth-of-good} 
Let $f: \Gamma \to \Gamma$ and $\lambda'$ be as in Convention \ref{tt-map-convention}. 
Let $\gamma$ be a reduced edge path or a cyclically reduced loop
in $\Gamma$. 
Then 
any good subpath $\gamma'$ of $\gamma$ has the property that no edge of $f(\gamma') = [f(\gamma')]$ is cancelled when $f(\gamma)$ is reduced, and that it consists entirely of edges that are good in $[f(\gamma)]$.
Hence, for any reduced loop $\gamma$ in $\Gamma$ we have:
$$\#\{{\rm good \,\,  edges \,\, in \,\,} [f(\gamma)]\}
\geq \lambda'  \cdot \#\{{\rm good \,\,  edges \,\, in \,\,} \gamma\}$$
\end{lem}

\begin{proof}
If $\gamma$ is legal, then every edge is good, and the lemma follows directly from the definition of $\lambda'$ in Convention \ref{tt-map-convention}. Now assume that the path $\gamma$ has at least one illegal turn. Let 
\[
\gamma=\gamma_1B_1\gamma_2B_2\ldots\ \gamma_nB_n
\]
be a decomposition of $\gamma$ into 
maximal 
good edge paths $\gamma_i$ and maximal 
bad edge paths $B_i$.

\begin{figure}[h!]
\labellist
\small\hair 2pt
\pinlabel {$\gamma_1$} [ ] at 230 585
\pinlabel {$a_1$} [ ] at 155 580
\pinlabel {$a_2$} [ ] at 310 580
\pinlabel {$b_2$} [ ] at 375 605

\endlabellist
\begin{center}
\includegraphics[scale=0.7]{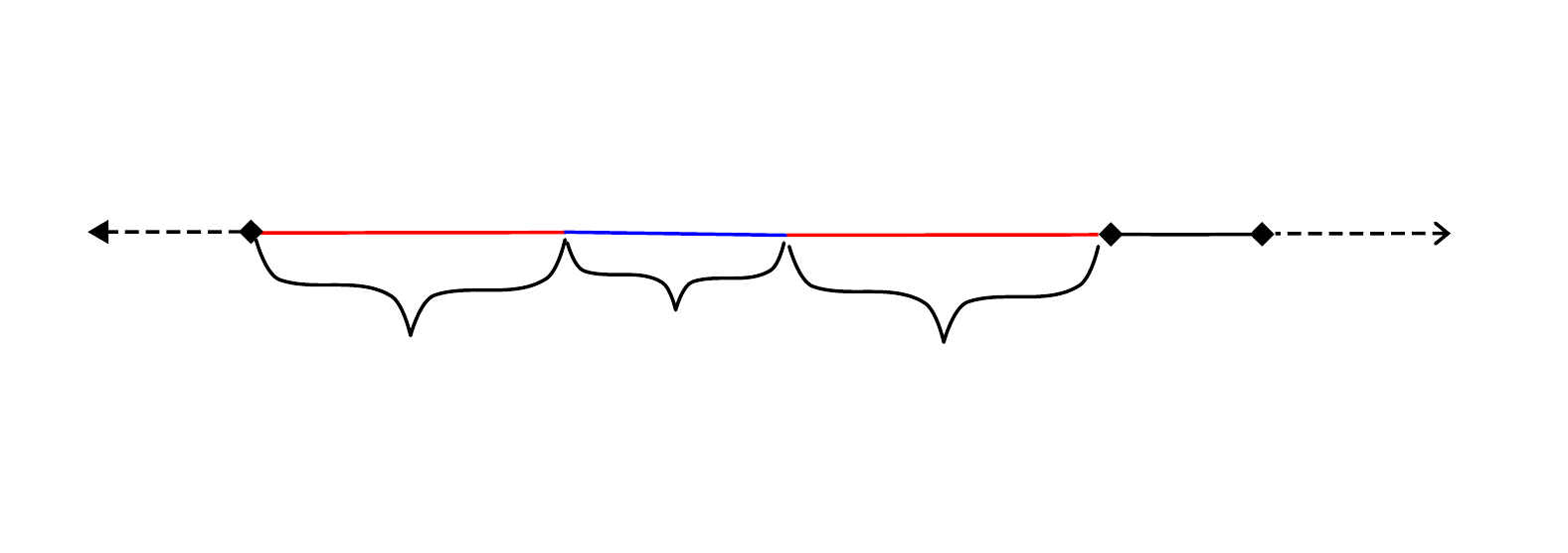}
\caption{$a_i$'s are legal ends of maximal bad segments, $\gamma$ is a good (legal) segment}
\label{goodedges}
\end{center}
\end{figure}
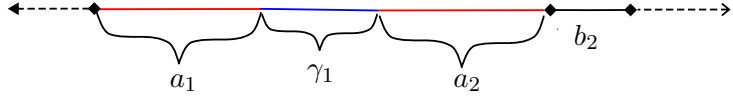

Note that each maximal bad segment $B_i$ can be written as an illegal concatenation $B_i=a_ib_ia_{i+1}$ 
where $a_i$ is a legal segment of length $C$ and $b_i$ is an edge path that (possibly) contains some illegal turns. 

Note that since $|f(a_i)|\ge\lambda'|a_i|\ge\lambda'C = C_f+C$, Lemma \ref{BCL} implies that $[f(B_i)]$ is an edge path of the from $[f(B_i)]=a_i'b_i'c_i'$ where $a_i',c_i'$ are legal edge paths such that $|a_i'|,|c_i'|\ge C$. Moreover, the turn at $f(\gamma_i)a_i'$ is legal. Since by Convention \ref{tt-map-convention} every edge grows at least by a factor of $\lambda'$, this implies the required result.  
\end{proof}

On the other hand, for any cyclically reduced loop $\gamma$ the number of bad edges is related to the number of illegal turns on $\gamma$, which we denote by $ILT(\gamma)$, via:
\begin{equation}
\label{illegal-bad}
ILT(\gamma)
\leq
\#\{{\rm bad \,\,  edges \,\, in \,\,} \gamma\}
\leq
2C \cdot ILT(\gamma)
\end{equation}
Since the number of illegal turns on $\gamma$ can only stay constant or decrease under iteration of the train track map, we obtain directly
\begin{equation}
\label{bad-edge-bound}
\#\{{\rm bad \,\,  edges \,\, in \,\,} [f^t(\gamma)]\}
\leq 2C \cdot ILT(\gamma)
\leq
2C \cdot \#\{{\rm bad \,\,  edges \,\, in \,\,} \gamma\}
\end{equation}
for all positive iterates $f^t$ of $f$. 

Notice however that the number of bad edges may actually grow (slightly) faster than 
the number of 
good edges under iteration of $f$, so that the goodness of $\gamma$ does not necessarily grow 
monotonically
under iteration of $f$. Nevertheless up to passing to powers one can overcome this issue:

\begin{prop}
\label{goodness-growth}
Let $f: \Gamma \to \Gamma$ be as in Convention \ref{tt-map-convention}. 

\smallskip
\noindent
(a)
There exists an integer $s \geq 1$ such that for every reduced loop $\gamma$ in $\Gamma$ one has:
\[
\g([f^s(\gamma)]) \geq \g(\gamma)
\]
In particular, for any integer $t \geq 0$ one has 
\[
\g([(f^s)^t(\gamma)]) \geq \g(\gamma)
\]

\smallskip
\noindent
(b)
If $0 < \g(\gamma) < 1$, then for any integer $s' > s$ we have
\[
\g([f^{s'}(\gamma)]) > \g(\gamma)
\]
and thus, for any $t \geq 0$:
\[
\g([(f^{s'})^t(\gamma)]) > \g(\gamma)
\]
\end{prop}

\begin{proof}
(a)
We set $s \geq 1$ so that ${\lambda'}^s \geq 2C$ and obtain from Lemma \ref{growth-of-good} for the number $g'$ of good edges in $[f^s(\gamma)]$ and the number $g$ of good edges in $\gamma$ that:
$$g' \geq {\lambda'}^s g$$
For the number $b'$ of bad edges in $[f^s(\gamma)]$ and the number $b$ of bad edges in $\gamma$ we have from equation (\ref{bad-edge-bound}) that:
$$b' \leq 2C b$$
Thus we get 
\[
\frac{g'}{b'} 
\geq
\frac{\lambda'^s}{2C}\frac{g}{b}
\geq \frac{g}{b}\ \text{and hence}\ \frac{g'}{g'+b'}\ge\frac{g}{g+b}
\]
which proves $\g([f^s(\gamma)]) \geq \g(\gamma)$. The second inequality in the statement of the lemma follows directly from an iterative application of the first.

\smallskip
\noindent
(b)
The proof of part (b) follows from the above given proof of part (a), since ${\lambda'}^{s} > 2C$ unless there is no good edge at all in $\gamma$, which is excluded by our hypothesis $0 < \g(\gamma) $. Since  the hypothesis $\g(\gamma) < 1$ implies that there is at least one illegal turn in $\gamma$, in the above proof we get $b\ge1$, which suffices to show
$$\frac{g'}{b'} > \frac{g}{b}\, ,$$
and thus $\g([f^s(\gamma)]) > \g(\gamma)$.
\end{proof}
From the inequalities at the end of part (a) of the above proof one derives directly, for $g > 0$, the inequality 
\[
\g ([f^s(\gamma)]) \geq \dfrac{1}{1+\frac{2C}{\lambda'^s}(\frac{1}{\g(\gamma)}-1)}. 
\]
Hence we obtain:

\begin{cor}
\label{goodbig}
Let $f: \Gamma \to \Gamma$ be as in Convention \ref{tt-map-convention}, and
let $\delta>0$ and $\epsilon>0$ be given. Then there exist an integer $M'=M'(\delta,\epsilon)\ge 0$ such that for any loop $\gamma$ in $\Gamma$ with $\g(\gamma)\geq\delta$ we have $\g([f^{m}(\gamma)])\geq1-\epsilon$ for all $m\ge M'$. 
\qed
\end{cor}

\medskip

\subsection{Illegal turns and iteration of the train track map}

${}^{}$
\smallskip

The following lemma (and also other statements of this subsection) are already known in differing train track dialects (compare for example Lemma 4.2.5 in \cite{BFH00}); for convenience of the reader we include here a short proof. Recall the definition of an INP and a pre-INP from 
Definition 
\ref{defn-INPs}. 

\begin{lem}
\label{finitely-many-INPs}
Let $f: \Gamma \to \Gamma$ be 
as in Convention \ref{tt-map-convention}. Then there are only finitely many INP's and pre-INP's in $\Gamma$ for $f$.  Furthermore, there is an efficient method to determine them.
\end{lem}

\begin{proof} We consider the set $\cal V$ of all pairs $(\gamma_1, \gamma_2)$ of legal edge paths $\gamma_1, \gamma_2$ with common initial vertex 
but distinct first edges, which have 
combinatorial length $|\gamma_1| = |\gamma_2| = C$. Note that $\cal V$ is finite.

From the definition of the cancellation bound and the inequalities (\ref{1.1})  it follows directly that every INP or every pre-INP $\eta$ must be a subpath of some path $\bar \gamma_1 \circ \gamma_2$ with $(\gamma_1, \gamma_2) \in \cal V$. Furthermore, we define for $i = 1$ and $i =2$ the initial subpath $\gamma^*_i$ of $\gamma_i$ to consist of all points $x$ of $\gamma_i$ that are mapped by some positive iterate $f^t$ into the backtracking subpath at the tip of the unreduced path $f^t(\bar \gamma_1 \circ \gamma_2)$. We observe that the interior of any INP-subpath or pre-INP-subpath $\eta$ of $\bar \gamma_1 \circ \gamma_2$ must agree with the subpath $\bar \gamma^*_1 \circ \gamma^*_2$.  Thus any pair $(\gamma_1, \gamma_2) \in \cal V$ can define at most one 
INP-subpath of $\bar \gamma_1 \circ \gamma_2$. Since $\cal V$ is finite and easily computable, we obtain directly the claim of the lemma.
\end{proof}

In light of Lemma \ref{finitely-many-INPs}, we obtain a \emph{subdivision} $\Gamma'$ of $\Gamma$ by adding the endpoints of 
all the 
finitely many INP's or pre-INP's, 
while keeping the property that $f$ maps vertices to vertices.
This gives  
a train track map $f':\Gamma' \to\Gamma'$ which represents the same outer automorphism as $f:\Gamma\to\Gamma$, thus justifying that from now on we concentrate on train track maps which satisfy the following:

\begin{conv}
\label{endpoints-of-INPs} 
Let $f: \Gamma \to \Gamma$ be a train track map as in
Convention \ref{tt-map-convention}. We assume in addition that 
every endpoint of an INP or a pre-INP is a vertex. 
\end{conv}

\begin{lem}
\label{two-INPs}
Let $f:\Gamma\to\Gamma$ be 
as in Convention 
\ref{endpoints-of-INPs}. 
Then, there exists an exponent $M_1 \geq 1$ with the following property:
Let $\gamma$ be a path in $\Gamma$, and assume that it contains precisely two illegal turns, 
which are the tips of 
INP-subpaths or pre-INP-subpaths $\eta_1$ and $\eta_2$ of $\gamma$.  If $\eta_1$ and $\eta_2$ overlap in a non-trivial subpath $\gamma'$, then $f^{M_1}(\gamma)$ reduces to a path $[f^{M_1}(\gamma)]$ which 
has at most one illegal turn.
\end{lem}

\begin{proof}
For any point $x_1$ in the interior of one of the two legal branches of an INP or a pre-INP $\eta$ there exists a point $x_2$ on the other legal branch 
such that for a suitable positive power of $f$ one has $f^t(x_1) = f^t(x_2)$. Hence, if we pick a point $x$ in the interior of $\gamma'$, there are points $x'$ on the other legal branch of $\eta_1$ and $x''$ on the other legal branch of $\eta_2$ such that for some positive iterate of $f$ one has $f^t(x') = f^t(x) = f^t(x'')$. It follows that the decomposition $\gamma = \gamma_1 \circ \gamma_2 \circ \gamma_3$, which uses $x'$ and $x''$ as concatenation points, defines legal subpaths $\gamma_1$ and $\gamma_3$ which yield $[f^t(\gamma)] = [f^t(\gamma_1) \circ f^t(\gamma_3)]$, which 
has at most one illegal turn.

Since by Convention \ref{endpoints-of-INPs} the overlap $\gamma'$ is an edge path, it follows from the finiteness result proved in Lemma \ref{finitely-many-INPs} that there are only finitely many constellations for $\eta_1$ and $\eta_2$. This shows that there must be a bound $M_1$ as claimed.
\end{proof}

\begin{lem}
\label{single-illegal-turn}
For every 
train track map $f: \Gamma \to \Gamma$ as in Convention 
\ref{endpoints-of-INPs}
there exists a constant $M_2= M_2(\Gamma) \geq 0$ such that every path $\gamma$ with precisely 1 illegal turn satisfies the following:

{Either $\gamma$ contains an INP or pre-INP as a subpath, or else $[f^{M_2}(\gamma)]$ is legal.}
\end{lem}

\begin{proof} 
Similar to the set $\cal V$ in the proof of Lemma \ref{finitely-many-INPs} we define the set $\cal V_+$ be the set of all pairs $(\gamma_1, \gamma_2)$ of legal edge paths $\gamma_1, \gamma_2$ in $\Gamma$ which have
combinatorial length $0 \leq |\gamma_1| \leq C$ and $0 \leq |\gamma_2| \leq C$, and which satisfy:

The paths $\gamma_1$ and $\gamma_2$ have common initial point, and, unless one of them (or both) are trivial, they have distinct first edges. Note that $\cal V_+$ is finite and contains $\cal V$ as subset.

\noindent Following ideas from \cite{CL15} we define a map 
$$f_\#: \cal V_+ \to \cal V_+$$
induced by $f$ through declaring the {\em $f_\#$-image} of a pair $(\gamma_1, \gamma_2) \in \cal V_+$ to be the pair $f_\#(\gamma_1, \gamma_2) := (\gamma''_1, \gamma''_2) \in \cal V_+$ which is defined by setting for $i \in \{1, 2 \}$
$$f(\gamma_i) =: \gamma'_i \circ \gamma''_i \circ \gamma'''_i\, ,$$
where $\gamma'_i$ is the maximal common initial subpath of $f(\gamma_1)$ and  $f(\gamma_2)$, where $|\gamma''_i| \leq  C$, and where $\gamma'''_i$ is non-trivial only if $|\gamma''_i| =  C$. 

Then for any $(\gamma_1, \gamma_2) \in \cal V_+$ we see as in the proof of Lemma \ref{finitely-many-INPs} that 
there exists an exponent $t \geq 0$ such that the iterate $f^t_\#(\gamma_1, \gamma_2) =: (\gamma_3, \gamma_4) \in \cal V_+$ satisfies one of the following:
\begin{enumerate}
\item
One of $\gamma_3$ or $\gamma_4$ (or both) are trivial.
\item
The turn defined by the two initial edges of $\gamma_3$ and $\gamma_4$ is legal.
\item
The path $\bar \gamma_3 \circ \gamma_4$ contains an INP as subpath.
\end{enumerate}
From the finiteness of $\cal V_+$ it follows directly that there is an upper bound $M_2 \geq 0$ such that for $t \geq M_2$ one of the above three alternatives must be true for $f^t_\#(\gamma_1, \gamma_2) =: (\gamma_3, \gamma_4)$.

Consider now the given path $\gamma$, and write it as illegal concatenation of two legal paths $\gamma = \gamma'_1 \circ \gamma'_2$. Then the maximal initial subpaths $\gamma_1$ of $\bar\gamma'_1$ and $\gamma_2$ of $\gamma'_2$ of $\gamma_2$ of length $|\gamma_i| \leq C$ form a pair $(\gamma_1, \gamma_2)$ in $\cal V_+$. In the above cases (1) or (2) it follows directly that $f^t(\gamma)$ is a (possibly trivial) legal path. In alternative (3) the path $f^t(\gamma)$ contains an INP.
\end{proof}

\begin{prop}
\label{ILT-decrease}
For any 
train track map as in Convention 
\ref{endpoints-of-INPs}
there exists an exponent 
$r = r(f) \geq 0$ such that every finite path $\gamma$ in $\Gamma$ 
with $\ILT(\gamma) \geq 1$ 
satisfies
$$\ILT([f^r(\gamma)]) < \ILT(\gamma) \, ,$$
unless every illegal turn on $\gamma$ is the tip of an INP or pre-INP subpath $\eta_i$ of $\gamma$, where any two $\eta_i$ are either disjoint subpaths on $\gamma$, or they overlap precisely in a common endpoint. 
\end{prop}

\begin{proof} 
Through considering maximal subpaths with precisely one illegal turn we obtain the claim as
direct consequence of Lemma \ref{single-illegal-turn} and Lemma \ref{two-INPs}.
\end{proof}

\begin{defn} 
\label{pseudo-leg}
A path $\gamma$ in $\Gamma$ is called \emph{pseudo-legal} if it is a legal concatenation of legal paths and INP's.
\end{defn}

From the same arguments 
as in the last proof 
we also deduce the following:

\begin{prop}
\label{pseudo-legal-iterate}
Let $f: \Gamma \to \Gamma$ be as in Convention 
\ref{endpoints-of-INPs}. Then
for 
any finite edge path 
$\gamma$ in $\Gamma$ there is a positive iterate $f^t(\gamma)$ which reduces to a path $\gamma' := [f^t(\gamma)]$ 
that 
is  pseudo-legal. The analogous statement also holds for loops 
instead of paths.
The exponent $t$ needed in either case depends only on the number of illegal turns in $\gamma$.
\end{prop}

\begin{proof}
Since the number of illegal turns in 
$[f^t(\gamma)]$ 
non-strictly decreases for increasing $t$, we can assume that for sufficiently large $t$ it stays constant. It follows from Lemma \ref{single-illegal-turn} (after possibly passing to a further power of $f$) that every illegal turn of 
$[f^t(\gamma)]$ 
is the tip of some INP-subpath $\eta_i$ of 
$[f^t(\gamma)]$. 
From Lemma \ref{two-INPs} we obtain furthermore that any two such $\eta_i$ and $\eta_j$ that are adjacent on 
$[f^t(\gamma)]$ 
can not overlap non-trivially.  
\end{proof}

In the next subsection we also need the following lemma, 
where 
{\em illegal (cyclic) concatenation}
means that the path (loop) is a concatenation of subpaths where all concatenation points are illegal turns.

\begin{lem} \label{pull-back-concatenation}
Let $f:\Gamma\to\Gamma$ be an expanding train track map. Let $\gamma$ be a reduced loop in $\Gamma$ and let $\gamma' = [f(\gamma)]$ be its reduced image loop. Assume that for some $t \geq 1$ a decomposition $\gamma' = \gamma'_1 \circ \gamma'_2 \circ  \ldots \circ \gamma'_t$ as illegal cyclic concatenation is given.  Then there exists a decomposition as illegal cyclic concatenation $\gamma = \gamma_1 \circ \gamma_2 \circ \ldots \circ \gamma_t$ with the property that the reduced image paths $[f(\gamma_i)]$ contain the paths $\gamma'_i$ as subpaths, for $i = 1, 2, \ldots, t$.
\end{lem}

\begin{proof}
We cut the loop $\gamma$ at some illegal turn to get a (closed) path $\gamma = e_1 e_2 \ldots e_q$. We consider the initial subpaths $\gamma(k) = e_1 e_2 \ldots e_k$ of $\gamma$ for $k = 1, 2, \ldots$. Let $k'$ be the smallest positive integer such that the reduced image path $[f(\gamma(k'))]$ contains the path $\gamma'_1$ as a subpath. Since $k'$ is the smallest such integer, it follows that the path $f(e_{k'})$ passes through the last edge of $\gamma'_1$. Note that $f(e_{k'})$ is a legal path
and that $\gamma'_i$ terminates at an illegal turn,
so that the 
endpoint 
of $f(e_{k'})$
must lie somewhere in the backtracking subpath of the possibly unreduced path $[f(e_1 e_2 \ldots e_{k')}] \circ [f(e_{k'+ 1} e_{k'+2} \ldots e_q)]$. It follows that the reduced image path $[f(e_{k'+ 1} e_{k'+2} \ldots e_q)]$ contains the path $\gamma'_2 \circ \ldots \circ \gamma'_t$ as subpath.

Thus we define $\gamma_1 := \gamma(k')$, and proceed iteratively in precisely the same fashion, thus finding iteratively $\gamma_2, \gamma_3, \ldots \gamma_{t-1}$. As above, it follows that the ``left-over" terminal subpath of $\gamma$ has the property that its reduced image contains $\gamma'_t$ as subpath, so that we can define this left-over subpath to be the final factor path $\gamma_t$.
\end{proof}


\subsection{Goodness versus illegal turns}

We recall from Definition \ref{defn-INPs}
that a multi-INP is a legal concatenation of finitely many INPs along their endpoints. For the remainder of this section we will concentrate on graph maps $f: \Gamma \to \Gamma$ which satisfy the following:

\begin{conv}
\label{conv-hyperbolic}
Let $f: \Gamma \to \Gamma$ be an expanding train track map as in 
Convention 
\ref{endpoints-of-INPs}.
We assume 
furthermore that 
there is an upper bound $A(f)$ to the number  of INP factor paths in any multi-INP path $\gamma\in\Gamma$. Note that this condition implies in particular that in $\Gamma$ there can not be a non-trivial loop which is a cyclic legal concatenation of INPs.
\end{conv}

The next 
lemma 
shows 
in particular 
that an expanding train track map 
which represents 
a hyperbolic outer automorphism naturally satisfies the conditions given in Convention \ref{conv-hyperbolic}. 

\begin{lem}
\label{INPfactors-} 
For any expanding train track map $f:\Gamma\to\Gamma$, 
which has the property that no positive power of $f$ fixes the homotopy class of any non-trivial loop in $\Gamma$,
there is an upper bound to the number of INP factors on any multi-INP path. 
\end{lem}

\begin{proof} 
There is an upper bound $A \geq 0$ for number of INPs in $\Gamma$, see Lemma \ref{finitely-many-INPs}. Hence, any multi-INP with more than $A(f) := 2A$ factors would have to run over the same INP twice in the same direction. Thus we obtain as subpath a non-trivial loop which is fixed by some positive power of $f$, violating the 
given 
assumption.
\end{proof}

\begin{lem} 
\label{badimage}
For any 
train track map $f:\Gamma\to\Gamma$ as in Convention \ref{conv-hyperbolic}
there exists an exponent $s\ge1$ such that for any reduced path $\gamma\in\Gamma$ with reduced image path $\gamma'=[f^s(\gamma)]$ the following holds: If $\g(\gamma')=0$ and satisfies $ILT(\gamma')\ge A(f)+1$, then 
\[
ILT(\gamma)>ILT(\gamma').
\]
\end{lem}

\begin{proof}
Since 
the map $f$ is expanding, for the critical constant 
$C \geq 0$ given in Definition \ref{goodness} there exists an integer $s\ge1$ such that $|f^s(e)|\ge2C+1$ for every edge $e\in\Gamma$. We can furthermore assume $s\ge r$, where $r=r(f)$ is given by Lemma \ref{ILT-decrease}.  Thus, by Lemma \ref{ILT-decrease} we have 
$$ILT(\gamma)>ILT(\gamma')\,,$$
unless every illegal turn on $\gamma$ is the tip of an INP or pre-INP $\eta_i$, and any two such subpaths $\eta_i$ are disjoint or they overlap precisely at a common endpoint. 

The case where two such paths are disjoint can be excluded as follows: If there is an edge $e$ in $\gamma$ which doesn't belong to any of the $\eta_i$, then $f^{s}(e)$ is a legal subpath of $[f^{s}(\gamma)]=\gamma'$ of length greater than $2C+1$, contradicting the assumption that $\g(\gamma')=0$. 

Thus we can assume that any two of the subpaths $\eta_i$ overlap precisely in a common endpoint, and that there is no non-trivial initial of final subpath of $\gamma$ outside of the concatenation of all the $\eta_i$.  Therefore, for some $s' \geq 0$, the iterate $f^{s'}(\gamma')$ is a multi-INP with $ILT(\gamma')\ge A(f)+1$ factors, which contradicts Convention \ref{conv-hyperbolic}. 
Hence the conclusion of the Lemma follows. 
\end{proof}

\begin{lem}
\label{goodnessthreshold}
Let $f:\Gamma\to\Gamma$ 
and $A(f)$ be as in Convention \ref{conv-hyperbolic}. 
Then there exists a constant $\delta$ with 
$0 < \delta \leq 1$
so that the following holds: Every reduced 
loop $\gamma$
in $\Gamma$ with $ILT(\gamma) \geq A(f)+1$ can be written as cyclic illegal concatenation
\[
\gamma=\gamma_1\circ\gamma_2\dotsc\circ\gamma_{2m}
\]
such that for every odd index $j$, the subpath $\gamma_j$ is either trivial or satisfies
\[
\g(\gamma_j)\ge\delta \, .
\] 
For every even index $k$ the subpath $\gamma_k$ is non-trivial and satisfies $\g(\gamma_k)=0$; moreover, we have:
\[
A(f)+1 \le ILT(\gamma_k)\le 2 A(f)+1
\]
\end{lem}

\begin{proof} Let ${L}_{\gamma}$ be the collection of maximal legal subpaths $\gamma'_i$ of $\gamma$ of length $|\gamma'_i| \ge 2C+1$, for $C \geq 0$ as given in Definition \ref{goodness}. Note that any two distinct elements $\gamma'_i, \gamma'_{j}$ 
in this collection are either disjoint or overlap at at single point. In the latter case the turn at the overlap point is illegal, as otherwise they would merge into a longer legal subpath, which violates the maximality of the $\gamma'_i$. 
For any two (on $\gamma$) consecutive paths $\gamma'_i$ and $\gamma'_{i+1}$ in $L_{\gamma}$, if the path $\beta_i$ between them is trivial or satisfies  $ILT(\beta_i)\le A(f)$, then we erase $\gamma'_i$ and $\gamma'_{i+1}$ from the collection ${L}_{\gamma}$ and add in the new path $\gamma'_i\beta_i\gamma'_{i+1}$. 
We continue this process iteratively until all the complementary subpaths $\beta_i$ between any two consecutive elements in our collection satisfies $ILT(\beta_i) \ge A(f)+1$. We call the obtained collection of subpaths $C_{\gamma}$. 

We now pick a path $\gamma_j$ in the collection $C_{\gamma}$ and consider its ``history'' as being obtained iteratively through joining what amounts to $\ell$ paths $\gamma'_i$ from $L_{\gamma}$. Thus $\gamma_j$ can be written as illegal concatenation 
$$\gamma_j = \gamma'_1 \circ \beta_1 \circ \gamma'_2 \circ \ldots \circ \gamma'_{\ell-1} \circ \beta_{\ell-1} \circ \gamma'_\ell,$$
where each $\gamma'_i$ is legal and of length $|\gamma'_i| \geq 2C+1$. Each $\beta_i$ has at most $A(f)$ illegal turns, and the legal subpaths of $\beta_i$ between these illegal turns have length $\leq 2C$, so that we get $|\beta_i| \leq  (A(f)+1) 2C$. Thus the set of good edges on $\gamma_j$ is given precisely as disjoint union over all the $\gamma'_i$ of the sets of the $|\gamma'_i|- 2C$ edges of $\gamma'_i$ that are not on the two boundary subpaths of length $C$ of $\gamma'_i$. Hence we compute for the goodness:
\[
\g(\gamma_j)
= \frac{\sum_{i=1}^{\ell} (|\gamma'_i|- 2C)}{|\gamma_j|}
= \frac{\sum_{i=1}^{\ell} (|\gamma'_i|- 2C)}{\sum_{i=1}^\ell |\gamma'_i| + \sum_{i=1}^{\ell - 1} |\beta_i| }
\]
\[
\ge\frac{\ell}{(2C+1) \cdot \ell+(A(f)+1)2C\cdot \ell}
= \frac{1}{2C (A(f)+2)+ 1}
\]

We finally add to the collection $C_{\gamma}$ a suitable set of trivial subpaths at illegal turns to get a collection $C'_{\gamma}$, where these trivial paths are chosen as to get as complementary subpaths of $C'_{\gamma}$ only subpaths $\gamma_k$ which satisfy:
\[
A(f)+1 \le ILT(\gamma_k)\le 
2 A(f)+1
\]

Thus setting 
\[
\delta=\frac{1}{2C (A(f)+2)+ 1}
\]
finishes the proof. 
\end{proof}

\begin{prop}
\label{dicho}
Let $f:\Gamma\to\Gamma$ be 
a 
train track map as in Convention \ref{conv-hyperbolic}, and
let $s\ge1$ be the integer given by Lemma \ref{badimage} and $\delta>0$ be the constant given by Lemma \ref{goodnessthreshold}. Then there exists a constant $R>1$ such that for any reduced loop $\gamma$ in $\Gamma$ 
one has 
either 
\[
(1) \quad
\g([f^{s}(\gamma)])\ge\frac{\delta}{2}
\]
or 
\[
(2) \quad 
ILT(\gamma)\ge R \cdot ILT([f^{s}(\gamma)]).
\]  
\end{prop}

\begin{proof}
By  
\ref{pseudo-legal-iterate} 
there is an exponent $t \geq 0$ such that for any loop $\gamma$ with less than $A(f)+1$ illegal turns the loop $\gamma' = [f^t(\gamma)]$ is pseudo-legal. From the 
assumption that $f$ satisfies Convention \ref{conv-hyperbolic}
it follows that $\gamma'$ is not a legal concatenation of INPs, so that it must have at least one good edge. Since iteration of $f$ only decreases the number of illegal turns, we obtain from equality (\ref{illegal-bad}) that 
$$\g(\gamma) \geq \frac{1}{2C(A(f)+1) +1} \, .$$
Thus the first inequality from our assertion follows for a suitable choice of $s$ from 
Proposition
\ref{goodness-growth}.

Thus we can now assume that $ILT(\gamma) \geq A(f)+1$, and thus that
$[f^{s}(\gamma)]=\gamma'_1\circ\dotsc\circ\gamma'_{2m}$ is a decomposition as given by Lemma \ref{goodnessthreshold}. There are two cases to consider:
Assume that 
\[
\sum_{j\ odd}|\gamma'_j|\ge\sum_{k\ even}|\gamma'_k|. 
\]
For  
any 
odd index $j$ 
the 
non-trivial path $\gamma'_j$ has $\g(\gamma'_j)\ge\delta$, which together with the  
last inequality 
implies:
\[
\g([f^{s}(\gamma)])\ge\frac{\delta}{2}
\]
Now, assume that: 
\[
\sum_{j\ odd}|\gamma'_j|\le\sum_{k\ even}|\gamma'_k|
\] 

Let $\gamma=\gamma_1\circ\dotsc\circ\gamma_{2m}$ be an illegal cyclic 
concatenation given by Lemma \ref{pull-back-concatenation} so that, for each $i=1,\dotsc, 2m$, the reduced image $[f^{s}(\gamma_i)]$ contains $\gamma'_i$ as a subpath.  Hence, for every even index, we apply Lemma \ref{badimage} to see that 
\[
ILT(\gamma_k)>ILT([f^{s}(\gamma_{k})])\ge ILT(\gamma'_k).
\] Since the number if illegal turns never increases when applying a train track map, for each odd index $j$, we have 
\[
ILT(\gamma_j)\ge ILT([f^{s}(\gamma_j)])\ge ILT(\gamma'_j).
\]
Combining last two inequalities we obtain:
\begin{equation}\label{eqn1}
ILT(\gamma)>ILT([f^{s}(\gamma)])+m.
\end{equation}

We also observe that 
the 
number of illegal turns in $[f^{s}(\gamma)]$ is equal to the sum of 
the
number of illegal turns in 
the 
odd indexed subpaths, 
the 
number of illegal turns in 
the 
even indexed subpaths, and 
the 
number of illegal turns at concatenation points:
\begin{equation}\label{eqn2}
ILT([f^{s}(\gamma)])\le\sum_{j\ odd}ILT(\gamma'_j)+\sum_{k\ even}ILT(\gamma'_k)+2m
\end{equation}
Now, note that 
\begin{equation}\label{eqn3}
\sum_{j\ odd}ILT(\gamma'_j)\le\sum_{j\ odd}|\gamma'_j|\le\sum_{k\ even}|\gamma'_k|,
\end{equation}
by assumption. 

For each $\gamma'_k$ (with even index $k$), since 
Lemma \ref{goodnessthreshold} assures
$\g(\gamma'_k)=0$ and $ILT(\gamma'_k)\le 2 (A(f)+1)$, we have 
\[
|\gamma'_k|\le 2C(2(A(f)+1)+1),
\]
which together with (\ref{eqn3}) implies that:
\begin{equation}\label{eqn4}
\sum_{j\ odd}ILT(\gamma'_j)\le 2mC(2(A(f)+1)+1) 
\end{equation}
Furthermore, from $ILT(\gamma'_k)\le 2 (A(f)+1)$ for even index $k$ we deduce
\begin{equation}\label{eqn4.5}
\sum_{k\ even}ILT(\gamma'_k)\le 2m(A(f)+1) 
\end{equation}

Using (\ref{eqn2}), (\ref{eqn4}) and (\ref{eqn4.5}) we obtain:
\begin{equation}\label{eqn5}
ILT([f^{s}(\gamma)])\le2mC(2(A(f)+1)+1)+2m(A(f)+1)+2m
\end{equation} 

Using (\ref{eqn1}) and (\ref{eqn5}), we have:
\begin{align*}
\frac{ILT(\gamma)}{ILT([f^{s}(\gamma)])}&\ge 1+\frac{m}{ILT([f^{s}(\gamma)])}\\
&\ge1+\frac{m}{2mC(2(A(f)+1)+1)+2m(A(f)+1)+2m}\\
&\ge1+\frac{1}{2C(2(A(f)+1)+1)+2(A(f)+1)+2}
\end{align*}

Therefore, the conclusion of the lemma holds for:
$$R=1+\dfrac{1}{2C(2(A(f)+1)+1)+2(A(f)+1)+2}$$

\end{proof}

\subsection{Uniform goodness in the future or the past}

For the convenience of the reader we first prove a mild generalization of Proposition \ref{dicho}, which will be a crucial ingredient in the proof of the next proposition. 

\begin{prop}
\label{dichonew}
Let $f:\Gamma\to\Gamma$ be 
as in Convention \ref{conv-hyperbolic}.
Given any constants  
$0 < \delta_1 < 1$ and 
$R_1>1$, 
there exist an integer $s_1>0$ so that for any reduced 
loops $\gamma$ and $\gamma'$ in $\Gamma$ with $[f^{s_1}(\gamma')] = \gamma$
one has either
\[
(i)\  \g([f^{s_1}(\gamma)])\ge
\delta_1
\]
or 
\[
(ii)\  ILT(\gamma')\ge R_1 \cdot |\gamma|_{\Gamma}\, .
\]  
\end{prop}

\begin{proof}
We first replace $f$ by a positive power (say $f^r$, cited at the end of the proof) as given by Proposition \ref{goodness-growth} so that for the rest of the proof we can assume that goodness is monotone.
Let $s, \delta$ and $R$ be as in Lemma \ref{dicho} and let $\gamma$ be any reduced loop in $\Gamma$. Assume that for $\gamma$
that alternative (1) of 
Lemma \ref{dicho} 
holds, i.e. $\g([f^{s}(\gamma)])\ge\frac{\delta}{2}$.  Then by 
Corollary \ref{goodbig} there is an exponent $M\ge1$ such that  
\[
\g([(f^{s})^{m}(\gamma)])\ge\delta_1
\]
for all $m\ge M$. 
On the other hand, if $\g([f^{s}(\gamma)])<\frac{\delta}{2}$, 
then 
Lemma \ref{dicho}  assures that 
\[
ILT(\gamma)\ge R \cdot ILT([f^{s}(\gamma)]). 
\]
We 
now 
claim that 
\[
ILT(
\gamma_1)
\ge R\cdot ILT(\gamma)
\] 
for any reduced loop $\gamma_1$ with 
$[f^{s}(\gamma_1)]  = \gamma$.
To see this, apply Lemma \ref{dicho} to the loop 
$\gamma_1$. 
If one had 
\[
\g([f^{s}(\gamma_1
)])=\g(\gamma)\ge\frac{\delta}{2}, 
\]
then by monotonicity 
of goodness this
 would also imply $\g([f^{s}(\gamma)])\ge\frac{\delta}{2}$, which contradicts with our assumption. Hence for 
 $\gamma_1$ alternative (2) of 
 Lemma \ref{dicho}  holds,
giving indeed  
\[
ILT(\gamma_1
)\ge R\cdot ILT(\gamma).
\]
Repeating the same argument iteratively shows that for any sequence of reduced  loops $\gamma_{M'}$, defined iteratively through $[f^s(\gamma_{M'})] = \gamma_{M'-1}$ for any $M' \geq 2$, the inequality
\[
\gamma_{M'}
\ge R^{M'}\cdot ILT(\gamma). 
\]
holds.
Also, notice that since $\g(\gamma)<\frac{\delta}{2}$ we have 
\[
\dfrac{\#\{{\rm bad \,\,  edges \,\, in \,\,} \gamma\}
}{|\gamma|_\Gamma}\ge1-\frac{\delta}{2},
\]
and 
by the inequalities (\ref{illegal-bad}) 
we have
furthermore
\[
\#\{{\rm bad \,\,  edges \,\, in \,\,} \gamma\}\le
2C \cdot ILT(\gamma).
\]
Hence we obtain 
\begin{align*}
ILT(\gamma_{M'})&\ge R^{M'}\cdot ILT(\gamma)\\
&\ge R^{M'}\cdot\frac{\#\{{\rm bad \,\,  edges \,\, in \,\,} \gamma\}}{2C}\\
&\ge R^{M'} \cdot (1-\frac{\delta}{2})\frac{1}{2C}|\gamma|_\Gamma
\end{align*}

Let $M' \geq 1$ 
be such that $R^{M'}\cdot(1-\frac{\delta}{2})\frac{1}{2C}\ge R_1$. Then $s_1=\max\{M,M'\}$ satisfies the requirements of the 
statement of Proposition \ref{dichonew}  
up to replacing $f$ by the power $f^r$ as done at the beginning of the proof.
\end{proof}

In addition to the train track map $f$ we now consider a 
second 
train track map 
$f':\Gamma'\to\Gamma'$,
which also
satisfies the requirements of 
Convention \ref{conv-hyperbolic} and 
which is related to $f$ via maps $h: \Gamma \to \Gamma'$ and $h': \Gamma' \to \Gamma$ such that $f$ and $h' f'h$ are homotopy inverses. 
We also assume that 
the 
lifts 
of $h$ and $h'$ 
to the universal coverings 
are quasi-isometries
(with respect to the simplicial metrics), 
so that there exists a bi-Lipschitz constant $B > 0$ which satisfies for any two ``corresponding'' reduced loops $\gamma$ in $\Gamma$ and $\gamma' := [h(\gamma)]$ the inequalities
\begin{equation}
\label{bi-Lipschitz}
\dfrac{1}{B}\,|\gamma'|_{\Gamma'}\,\,\le\,\,|\gamma|_{\Gamma}\,\,\le\,\, B \,|\gamma'|_{\Gamma'}
\end{equation}
We denote the goodness for the map $f'$ by $\g'$, and the 
critical 
constant for $f'$ from Definition \ref{goodness} by $C'$.

\begin{prop}\label{backforthgoodness}
Let $f: \Gamma \to \Gamma$ and $f':\Gamma' \to \Gamma'$ be as specified in the last paragraph. 

Given a real number $\delta$ so that $0<\delta<1$, there exist a bound $M>0$ such that (up to replacing $f$ and $f'$ by a common power) for any pair of corresponding reduced 
loops $\gamma$ in $\Gamma$ and $\gamma'$ in $\Gamma'$, either 
\[
(a)\  \g([f^n(\gamma)])\ge\delta
\]
or 
\[
(b)\ \g'([(f')^n(\gamma')])\ge\delta
\] 
holds for all $n\ge M$. 
\end{prop}

\begin{proof} Let $B$ be a bi-Lipschitz constant for the transition from $\Gamma$ to $\Gamma'$ as in (\ref{bi-Lipschitz}) above.
Set $R_1=4C'B^{2}$, and apply Proposition \ref{dichonew} to the loop $\gamma$. 
Assume first that alternative (i) of this proposition holds.  Then Corollary 
\ref{goodbig}, applied to $f^{s_1}$, gives 
a bound $L \geq 0$ so that inequality (a) holds for all $n \geq L$ (after having replaced $f$ by $f^{s_1}$).

Now assume that for $\gamma$ alternative (ii) of Proposition \ref{dichonew} holds. Then we have the following inequalities, where $\gamma''$ denotes the reduced loop in $\Gamma$ corresponding to $f'^{s_1}(\gamma')$ (which implies 
$[f^{s_1}(\gamma'')] = \gamma$):
\begin{align*} 
|f'^{s_1}(\gamma')|_{\Gamma'}&
\ge\frac{1}{B}|\gamma''|_{\Gamma}
\ge\frac{1}{B} ILT(\gamma'')
\\
&\ge 
\frac{1}{B}R_1\cdot |\gamma|_{\Gamma}=4C'B\cdot |\gamma|_{\Gamma}\\
&= 4C'\cdot |\gamma'|_{\Gamma'}
\end{align*}
This however implies that $\g'(f'^{s_1}(\gamma'))\ge1/2$, since for any $t \geq 0$  the number of bad edges in $f'^{t}(\gamma')$ 
is bounded above by 
\[
2C' \cdot ILT(f'^t(\gamma')) \leq 2C' \cdot ILT(\gamma') \leq 2C'\cdot |\gamma'|_{\Gamma'}.
\] Thus, by invoking Corollary \ref{goodbig} again, there exist $L'$ such that \[
\g'(f'^{s_1L'}(\gamma'))\ge\delta.
\]
Hence for $M'=\max\{L,L'\}$ the conclusion of the Lemma follows. 
\end{proof}

\begin{rem}
\label{back-to-ht}
From the next section on we will restrict our attention to 
expanding train track maps which are homotopy equivalences. For any such train track representative $f$ of some $\phi \in \Out(\FN)$ and $f'$ of $\phi^{-1}$, which both satisfy Convention \ref{conv-hyperbolic}, the conditions stated before Proposition \ref{backforthgoodness} are satisfied, so that this proposition is valid for them.
\end{rem}

\section{Perron Frobenius Theorem for reducible substitutions and automorphisms}
\label{symbolic}

The 
purpose of this section is to define the \emph{simplex of attraction} $\Delta_{+}$ and the 
\emph{simplex of repulsion} $\Delta_{-}$ explicitly by exhibiting the coordinate values for the defining currents. 
For this purpose we crucially need the main result from our earlier paper \cite{LU1}; details of how to use this quote are given below.

\subsection{Symbolic dynamics, substitutions and quotes from \cite{LU1}}

In this subsection we 
recall some terminology, well known definitions and basic facts from symbolic dynamics. 

Given any finite set of symbols $A = \{a_1, \ldots, a_m\}$, called an {\em alphabet},
one considers the free monoid $A^*$ over $A$, which is given by all finite words in $A$ (without inverses), including the empty word $1$.  The \emph{length} $|w|$ of a word $w\in A^*$ denotes the number of letters in $w$. 
Furthermore one considers the {\em full shift} $A^\Z$, the set of all biinfinite words in $A$, which is canonically equipped with the product topology and with the shift-operator. The latter maps $\underbar x = \ldots x_{i-1} x_i x_{i+1} \ldots \in A^\Z$ to $\underbar y = \ldots y_{i-1} y_i y_{i+1} \ldots \in A^\Z$, with $y_i = x_{i+1} \in A$ for any $i \in \Z$.
A subset $\Sigma \subset A^\Z$ is called a {\em subshift} if it is 
non-empty, closed and shift-invariant. 
The \emph{language} 
of $\Sigma$ is the collection of all finite words occurring in elements of $\Sigma$. 
For any $w = x_1 \ldots x_n \in A^*$ we denote by ${\rm Cyl}_w$ the set of all $\underbar y = \ldots y_{i-1} y_i y_{i+1} \ldots \in A^\Z$ which satisfy $y_1 = x_1, y_2 = x_2, \ldots, y_n = x_n$. 

\smallskip

An {\em invariant measure} $\mu$ on a subshift $\Sigma$ is a finite Borel measure on $\Sigma$ which is invariant under the shift map. It extends canonically to an invariant measure on all of $A^\Z$ by setting $\mu(B) := \mu(B \cap \Sigma)$ for any measurable set $B \subset A^\Z$. For any invariant measure $\mu$ on $A^\Z$ obtained this way we say that $\mu$ has {\em support} in $\Sigma$.

\begin{rem}
\label{Kirch}
(1)
Any invariant measure $\mu$ defines a function 
$$\mu_A: A^* \to \R_{\geq 0}, \, w \mapsto \mu({\rm Cyl}_w)$$
which satisfies for any $w \in A^*$ the following {\em Kirchhoff conditions}:
$$\mu_A(w) = \sum_{x \in A} \mu_A(xw) = \sum_{x \in A} \mu_A(wx)$$ 
Conversely, for every such {\em weight function} $\mu^*: A^* \to \R_{\geq 0}$  (i.e. $\mu^*$ satisfies the Kirchhoff conditions) there exists a well defined invariant measure $\mu$ on $A^\Z$ with $\mu_A = \mu^*$, see \cite {FM}. 

\smallskip
\noindent
(2)
The measure $\mu$ has support in a subshift $\Sigma \subset A^\Z$ if and only if $\mu_A(w) = 0$ for any $w \in A^*$ with ${\rm Cyl}_w \cap \Sigma = \emptyset$.
\end{rem}

A \emph{substitution} $\zeta$ on the finite alphabet $A=\{a_1,\ldots a_m\}$ is 
a monoid morphism $\zeta: \A^* \to A^*$. It is well defined by knowing all $\zeta(a_i)$, and conversely, any map $A \to A^*$ defines a substitution. A substitution $\zeta:A\to A^*$ is called \emph{expanding} if for each $a_i\in A$, $|\zeta^k(a_i)|\ge2$ for some $k\ge1$. For any two words $w_1,w_2\in A^*$, let us denote the number of occurrences of the word $w_1$ in $w_2$ by $|w_2|_{w_1}$.  
By proving a generalized version of the classical Perron--Frobenius Theorem for reducible matrices, in \cite{LU1} we obtain the following result. 


\begin{prop}[Corollary 1.3 and Remark 3.14 of \cite{LU1}]
\label{invariant-measure1}
For any expanding substitution $\zeta: A^* \to A^*$ and any letter 
$a_i \in A$ there is a well defined invariant measure $\mu_{a_i}$ on $A^\Z$.
For any 
non-empty $w \in A^*$ and the associated 
cylinder ${\rm Cyl}_w$
the value of $\mu_{a_i}$
is
given, 
after possibly raising $\zeta$ to a suitable power (which does not depend on $w$), 
by the limit frequency 
\[
\mu_{a_i}({\rm Cyl}_w) = \lim_{t\to\infty}\frac{|\zeta^{t}(a_i)|_w}{|\zeta^{t}(a_i)|}.
\]
\end{prop}

\begin{rem}
\label{transition}
(1) We note that in \cite{LU1} the measure $\mu_{a_i}$ is defined on the substitution subshift $\Sigma_\zeta$, while in the above quote we denote by $\mu_{a_i}$ the canonical extension to all of $A^\Z$. Accordingly, in \cite{LU1} the set ${\rm Cyl}_w$ 
denotes 
set ${\rm Cyl}_w \cap \Sigma_\zeta$.
However, 
as explained in Remark 3.13 of \cite{LU1}, these differences are only superficial: the statement quoted above and the one in \cite{LU1} are indeed equivalent.

\noindent
(2)
It can be derived from \cite{LU1} that the positive power to which $\zeta$ has to be raised in order to get the frequency equality in Proposition \ref{invariant-measure1} is not just independent of $w$ but actually independent of $\zeta$. It only depends on the size of the alphabet $A$.
\end{rem}

In the next subsection we also need the following:

\begin{lem}
[Remark 3.3 of 
\cite{LU1}]
\label{eigenvalue-}
Let $\zeta: A^* \to A^*$ be  an expanding substitution. Then (up to replacing $\zeta$ by a positive power) there exists a constant $\lambda_{a_i} >1$ for each $a_i \in A$ such that:
$$\lim_{t \to \infty} \frac{|\zeta^{t+1}(a_i)|}{|\zeta^{t}(a_i)|} = \lambda_{a_i}$$
\qed
\end{lem}

\subsection{Train track maps reinterpreted as substitutions}

We now apply the material from the previous subsection to a particular substitution $\zeta_f$ which we construct from an expanding train track map $f: \Gamma \to \Gamma$ 
as in the section \ref{train-tracks}.  We first set up the following general definitions:

For any graph $\Gamma$ we recall (see section \ref{graphs}) that the set $E\Gamma$ of oriented edges of $\Gamma$ is equipped with a fixed point free involution 
\[
\iota: E\Gamma \to E\Gamma, \,\, e \mapsto \bar e
\]
where 
$\bar e$ 
denotes as before the edge $e$ with reversed orientation.

We now use $E\Gamma$ as alphabet, and consider the free monoid $E\Gamma^*$ and the 
both-sided full shift $E \Gamma^\Z$. The latter contains the subshift (of ``finite type'') $\Sigma_\Gamma \subset E \Gamma^\Z$, which consists of 
all biinfinite reduced edge paths in $\Gamma$. We denote the language of $\Sigma_\Gamma$ by 
$\cal P(\Gamma)$, it is the set  
of all finite reduced edge paths in $\Gamma$.
We observe that the involution $\iota$ induces an involution on $E \Gamma^*$ 
which for convenience will be denoted also by $\iota$.
On $E\Gamma^\Z$ this involution induces the {\em flip involution} which is given by 
$$ \iota( \ldots x_i x_{i+1} x_{i+2} \ldots) =  \ldots y_i y_{i+1} y_{i+2} \ldots$$
with $y_n = \bar x_{-n+1}$ for all $n \in \Z$. The subshift $\Sigma_\Gamma$ is clearly flip-invariant.

A shift-invariant measure $\mu^*$ on $\Sigma_\Gamma$ is called {\em flip-invariant} if for any measurable subset $\Sigma' \subset \Sigma_\Gamma$ one has $\mu^*(\iota(\Sigma')) = \mu^*(\Sigma')$. 
From Remark \ref{Kirch} we see that this is equivalent to requiring that the weight function $\mu^*_{E\Gamma}$ associated to $\mu^*$ satisfies
$$\mu^*_{E\Gamma}(\bar \gamma) = \mu^*_{E\Gamma}(\gamma)$$
for any $\gamma \in \cal P(\Gamma)$. 
We now observe:

\begin{lem}
\label{measure-current}
Let $\Gamma$ be a graph provided with a marking $\theta: \FN \to \pi_1\Gamma$. Then any shift-invariant and flip-invariant measure $\mu^*$ on $\Sigma_\Gamma$ defines a current $\mu$ on $\FN$ and conversely. This canonical 1-1 correspondence is given by the fact that the 
Kolmogorov function $\mu_{E\Gamma}$ associated to $\mu$ 
(see section \ref{sec:currents}) and the weight function $\mu^*_{E\Gamma}$ associated to $\mu^*$ satisfy
$$\mu_{E\Gamma}(\gamma) = \mu^*_{E\Gamma}(\gamma)$$
for any finite reduced path $\gamma$ in $\Gamma$.
\end{lem}

\begin{proof}
It is well known that a current and the associated Kolmogorov function determine each other vice versa (see section \ref{sec:currents}), and similarly for  a shift-invariant measure and its weight function (see Remark \ref{Kirch}).
Hence the statement of the lemma is a direct consequence of the observation that any flip-invariant weight function on $E\Gamma^*$, which is zero on $E\Gamma^* \smallsetminus \cal P(\Gamma)$, restricts on $\cal P(\Gamma)$ to a Kolmogorov function, and conversely, any Kolmogorov function on $\cal P(\Gamma)$ gives canonically rise to a flip-invariant weight function on $E\Gamma$ by setting it equal to $0$ for any element of on $E\Gamma^* \smallsetminus \cal P(\Gamma)$.
\end{proof}

\medskip

We now consider the occurrences of a path $\gamma$ as subpath in a path $\gamma'$. As before, 
the number of such occurrences is denoted by $|\gamma'|_\gamma$. We denote the number of occurrences of $\gamma$ or of $\bar \gamma$ as subpath in a path $\gamma'$ by $\langle \gamma, \gamma' \rangle$ and obtain:
\begin{equation}
\label{occurrences}
\langle \gamma, \gamma' \rangle = |\gamma'|_\gamma + |\gamma'|_{\bar \gamma}
\end{equation}

\medskip

Let now $f: \Gamma \to \Gamma$ be any graph self-map. Then $f$ induces a substitution $\zeta_f: E\Gamma^* \to E\Gamma^*$, but in general $\zeta_f$-iterates of reduced paths in $\Gamma$ will be mapped to non-reduced paths. An exception is given if $f$ is a train track map and if the path $\gamma'$ is legal: In this case all $f^t(\gamma')$ will be reduced as well: 
$$[f^t(\gamma')] = f^t(\gamma') \,$$
where $[\rho]$ denotes as in section \ref{graphs} the path obtained from an edge path $\rho$ via reduction relative to its endpoints.
Hence the occurrences of any path $\gamma$ or of $\bar \gamma$ in $[f^t(\gamma')]$ are given by 
\begin{equation}
\label{occurrences+}
\langle \gamma, [f^t(\gamma')] \rangle = |f^t(\gamma')|_\gamma + |f^t(\gamma')|_{\bar \gamma}
\end{equation}
for any integer $t \geq 0$.

We are now ready to prove:

\begin{prop} 
\label{edgeconvergence} 
Let 
$f:\Gamma\to\Gamma$ be
an expanding 
train-track map,
and
let $e\in E\Gamma$. 
Then, 
after possibly replacing $f$ by a 
positive power, 
for any reduced edge path $\gamma$ in $\Gamma$ 
the limit
\[
\lim_{n\to\infty}\frac{\langle\gamma,f^{t}(e)\rangle}{|f^{t}(e)|} = a_{\gamma}
\]
exists, 
and
the set of these limit values 
defines a unique geodesic current $\mu_{+}(e)$ on $F_N$
through setting $\langle \gamma, \mu_+(e) \rangle = a_\gamma$ for any $\gamma \in \cal P(\Gamma)$. 
\end{prop}

\begin{proof}
Since every edge $e$ of $\Gamma$ is legal, we obtain for any $\gamma \in \cal P(\Gamma)$ from
equation (\ref{occurrences+}) above that  
$\langle \gamma, f^t(e) \rangle = \langle \gamma, [f^t(e)] \rangle = |f^t(e)|_\gamma + |f^t(e)|_{\bar \gamma}$ for all exponents $t \geq 0$. 
On the other hand, for any element $\omega \in E\Gamma^* \smallsetminus \cal P(\Gamma)$ one has $|f^t(e)|_\omega + |f^t(e)|_{\bar \omega} = 0$.

Furthermore the assumption that $f$ is expanding implies directly that the substitution $\zeta_f$ associated to $f$ is also expanding.

Hence we can apply Proposition \ref{invariant-measure1} to the substitution $\zeta_f$ to obtain that, 
after possibly replacing $f$ and $\zeta_f$ by a suitable positive power, 
the limit of the $\frac{\langle\gamma,f^{t}(e)\rangle}{|f^{t}(e)|}$ exists indeed, and that it is equal to $\mu^*_{e}({\rm Cyl}_\gamma) + \mu^*_{e}({\rm Cyl}_{\bar \gamma})$, where $\mu^*_e$ is a shift-invariant measure on $E\Gamma^*$ with support in
the subshift 
$\Sigma_\Gamma$. 
Hence $\gamma \mapsto \mu^*_{e}({\rm Cyl}_\gamma)$ defines a weight function on $E\Gamma^*$ 
which is zero outside $\cal P(\Gamma)$, and thus 
similarly 
$\gamma \mapsto \mu^*_{e}({\rm Cyl}_{\bar \gamma})$. Their sum is thus also a weight function
with support in $\cal P(\Gamma)$, 
and furthermore it is flip invariant, so that we obtain from Lemma \ref{measure-current} a Kolmogorov function. This in turn defines the current $\mu_+(e)$, as stated in our proposition.
\end{proof}

We now want to show that the currents $\mu_+(e)$ are projectively $\phi$-invariant. For this purpose we start by stating two lemmas; the first one is an elementary exercise:

\begin{lem}
\label{closing-up}
For any graph $\Gamma$ 
without valence 1 vertices
there exists a constant $K \geq 0$ such that for any finite reduced edge path $\gamma$ in $\Gamma$ there exists an edge path $\gamma'$ of length $|\gamma'| \leq K$ such that the concatenation $\gamma \circ \gamma'$ exists and is a reduced loop.
\qed
\end{lem}

\begin{lem}
\label{K-approximation}
Let $f: \Gamma \to \Gamma$ as in Proposition \ref{edgeconvergence}, and let $K_1 \geq 0$ be any constant.
For all 
integers $t \geq 0$ let $\gamma'_t \in E\Gamma^*$ be any element with $|\gamma'_t| \leq K_1$. Set $\gamma_t := f^t(e)^* \gamma'_t$, where $f^t(e)^*$ is obtained from $f^t(e)$ by erasing an initial and a terminal subpath of length at most $K_1$. Then
for any reduced path $\gamma$ in $\Gamma$ one has 
$$\underset{t \to \infty}{\lim} \frac{\langle \gamma, \gamma_t  \rangle}{|\gamma_t|} = \langle \gamma, \mu_+(e) \rangle$$
\end{lem}

\begin{proof}
From the hypotheses $|\gamma'_t| \leq K_1$ and $|f^t(e)^*| \geq |f^t(e)| - 2K_1$, and from the fact that $f$ is expanding and hence $|f^t(e)| \to \infty$ for $t \to \infty$, we obtain directly 
$\underset{t \to \infty}{\lim} \frac{\langle \gamma, \gamma_t  \rangle}{\langle \gamma, f^t(e)  \rangle} = 1$ and
$\underset{t \to \infty}{\lim} \frac{|\gamma_t|}{| f^t(e)  |} = 1$. Hence the claim follows directly from Proposition \ref{edgeconvergence}.
\end{proof}

\begin{prop}
\label{fixed-current}
Let $\varphi\in Out(F_N)$ be an 
automorphism which is represented by an expanding 
train-track map $f:\Gamma\to\Gamma$. 
We assume that $\phi$ and $f$ have been replaced by positive powers according to Proposition \ref{edgeconvergence}.

Then there exist a constant $\lambda_{e}>1$ such that $\varphi(\mu_{+}(e))=\lambda_{e}\mu_{+}(e)$. 
\end{prop}

\begin{proof}
For the given graph $\Gamma$ let $K \geq 0$ be the constant given by Lemma \ref{closing-up}, 
and for any integer $t \geq 0$ let $\gamma'_t \in \cal P(\Gamma)$ with $|\gamma'_t| \leq K$ be the path given by Lemma \ref{closing-up} so that $\gamma_t =: f^t(e) \gamma'_t \in \cal P(\Gamma)$ is a reduced loop. Let $[w_t] \subset \FN \cong \pi_1\Gamma$ be the conjugacy class represented by $\gamma_t$, and note that the rational current $\eta_{[w_t]}$ satisfies $\|\eta_{[w_t]}\| = |\gamma_t|$, by equality (\ref{current-norm}) from section \ref{sec:currents}.

Similarly, consider $f(\gamma_n) = f^{t+1}(e) f(\gamma'_t)$, and notice that $|f(\gamma'_t)|$ is bounded above by the constant $K_0 = K \, \max\{|f(e)| \mid e \in E\Gamma\}$. Since $f$ is a train track map, the path $f^{t+1}(e)$ is reduced. Hence the reduced loop $\gamma''_t := [f(\gamma_n)] = [f^{t+1}(e) f(\gamma'_t)]$ can be written as product $f^{t+1}(e)^* \gamma''_t$ with 
$|\gamma''_t| \leq K_1$ and $|f^{t+1}(e)^*| \geq |f^{t+1}(e)| - 2K_1$, where $f^{t+1}(e)^*$ is a subpath of $f^{t+1}(e)$ and $K_1$ is the maximum of $K_0$ and the cancellation bound $C_f$ of $f$ (see Lemma \ref{BCL}).

Thus we can apply Lemma \ref{K-approximation} twice to obtain for any reduced path $\gamma$ in $\Gamma$ that
$$\underset{t \to \infty}{\lim} \frac{\langle \gamma, \gamma_t  \rangle}{|\gamma_t|} = \langle \gamma, \mu_+(e) \rangle$$
and
$$\underset{t \to \infty}{\lim} \frac{\langle \gamma, \gamma''_t  \rangle}{|\gamma''_t|} = 
\langle \gamma, \mu_+(e) \rangle \, .$$
The first equality implies that the rational currents $\eta_{[w_t]}$ satisfy \[
\underset{t \to \infty}{\lim}\frac{\eta_{[w_t]}}{\|\eta_{[w_t]}\|} = \mu_+(e).
\]
From the continuity of the $\Out(\FN)$-action on current space and from $\phi \eta_{[w_t]}= \eta_{\phi[w_t]}$ (see equality (\ref{rational-image}) from section \ref{sec:currents}) we thus deduce 
\[
\underset{t \to \infty}{\lim}\frac{\eta_{\phi[w_t]}}{\|\eta_{[w_t]}\|} = \phi \mu_+(e).
\]

However, since the reduced loops $\gamma''_t$ represent the conjugacy classes $\phi[w_t]$,  the second of the above equalities implies that 
\[
\underset{t \to \infty}{\lim}\frac{\eta_{\phi[w_t]}}{\|\eta_{\phi[w_t]}\|} = \mu_+(e).
\] 
Since $\underset{t \to \infty}{\lim} \frac{|\gamma_t|}{| f^t(e)  |} = 1$ and $\underset{t \to \infty}{\lim} \frac{|\gamma''_t|}{| f^{t+1}(e)  |} = 1$, with $|\gamma_t| = \|\eta_{[w_t]}\|$ and $|\gamma''_t| = \|\eta_{\phi[w_t]}\|$, 
the conclusion follows from Lemma \ref{eigenvalue+} below.
\end{proof}

\begin{lem}
\label{eigenvalue+}
For every edge $e$ of $\Gamma$ there exists a real number $\lambda_e > 1$ which satisfies:
$$\lim_{t \to \infty} \frac{|f^{t+1}(e)|}{|f^{t}(e)|} = \lambda_e$$
\end{lem}

\begin{proof}
This is a direct consequence of Lemma \ref{eigenvalue-} and of the definition of $\zeta_f$.
\end{proof}

\bigskip

We now define $\Delta_{-}(\varphi)$ and $\Delta_{+}(\varphi)$ that are used in the next section:

\begin{defn} 
\label{inv-simpl}
Let $\varphi\in Out(F_N)$ be an outer automorphism. Assume that $\varphi$ is replaced by a  
positive power such that both, $\varphi$ and $\phi^{-1}$ are represented by expanding train-track maps
as in Proposition \ref{edgeconvergence}. Let $f:\Gamma\to\Gamma$ be the representative of $\phi$.
Then the 
{\em simplex of attraction} 
is defined as follows: 
\[
\Delta_{+}
(\phi)=\{[\sum_{e_i\in E^{+}\Gamma}a_i\mu_{+}(e_i)]\mid a_i\ge0, \sum a_i>0\}. 
\]
Analogously, we define the 
{\em simplex of repulsion} 
as $\Delta_{-}(\varphi)=\Delta_{+}(\varphi^{-1})$. 
\end{defn}

It should be pointed out here that distinct edges $e_i$ of $\Gamma$ may well give the same limit current $\mu_+(e_i)$. Furthermore, even if a subset of currents $\mu_+(e_i)$
which define 
pairwise different  vertices of the ``simplex'' $\Delta_+(\phi)$, this subset 
may 
a priori 
not be linearly independent. 
However, it follows from general results in dynamics (see for instance \cite{FM}) 
that the extremal points of $\Delta_+(\phi)$ correspond to ergodic currents, and that pairwise projectively distinct ergodic currents 
are linearly independent, 
so that $\Delta_+(\phi)$ is indeed a finite dimensional simplex.

From Proposition \ref{fixed-current} we obtain directly:

\begin{cor}
\label{delta-invariance}
Let $\phi \in \Out(\FN)$ be as in Definition \ref{inv-simpl}. Then the simplexes $\Delta_{+}(\varphi)$ and $\Delta_{-}(\varphi)$ are $\phi$-invariant:
$$
\phi(\Delta_{+}(\varphi)) = \Delta_{+}(\varphi) \qquad \text{and} \qquad  \phi(\Delta_{-}(\varphi)) = \Delta_{-}(\varphi)
$$
\qed
\end{cor}

\begin{rem}
\label{max-dim}
There seems to be an interesting question 
as to 
what the maximal possible dimension of the limit simplexes $\Delta_+(\phi)$ (or $\Delta_-(\phi) = \delta_+(\phi^{-1})$) can be. It has been shown by G. Levitt 
\cite{Levitt} that for any $\phi \in \Out(\FN)$ the number of exponentially growing strata in any 
train track representative of $\phi$ is bounded above by $\frac{3N-2}{4}$ (and that this bound is attained, by certain geometric automorphisms). This gives: 
$$\dim(\Delta_+(\phi)) \leq \frac{3N-6}{4}$$
We do not know whether there are hyperbolic automorphisms which realize this bound. Preliminary considerations have lead us to construct for any integer $k \geq 1$ a family of hyperbolic automorphisms $\phi_k$ with $\text{rank}(\phi_k) 
=  2 \cdot 3^k + 3$ and $\dim(\Delta_+(\phi_k)) 
= \frac{3^{k+1} - 1}{2}$.

\end{rem}

In the next section we will specify the automorphism in question to be hyperbolic. Recall that $\phi \in \Out(\FN)$ is hyperbolic if there is no non-trivial $\phi$-periodic conjugacy class in $\FN$. Let $f:\Gamma \to \Gamma$ be an absolute train track representative of $\phi$. By Remark \ref{make-expanding} we can assume that $f$ is expanding. 
We then replace $\phi$ and $f$ by a positive power so that Convention \ref{tt-map-convention} is satisfied, and furthermore subdivide edges in accordance to Convention \ref{endpoints-of-INPs}.
Furthermore we have:

\begin{lem}
\label{INPfactors} 
Any expanding train track representative $f:\Gamma\to\Gamma$ of a hyperbolic automorphism 
$\varphi\in Out(F_N )$
which satisfies Convention \ref{endpoints-of-INPs} also 
satisfies the hypotheses of Convention \ref{conv-hyperbolic}.
\end{lem}

\begin{proof} 
This is a direct consequence of Lemma \ref{INPfactors-}.
\end{proof} 

\begin{conv}
\label{before-section-6}
A hyperbolic automorphism $\phi \in \Out(\FN)$ is said to have a {\em properly expanding} train track representative $f: \Gamma \to \Gamma$ if 
$f$ satisfies Conventions \ref{tt-map-convention},  \ref{endpoints-of-INPs} and \ref{conv-hyperbolic}, and if both have been raised to a suitable positive power according to Propositions \ref{edgeconvergence}  and \ref{fixed-current}.
\end{conv}

\begin{rem}
\label{proper-tts-exist}
From Lemma \ref{INPfactors} one obtains directly   
that every hyperbolic automorphism $\phi$ which possesses an absolute train track representative has a positive power $\phi^t$ which possesses a properly expanding train track representative. 
\end{rem}

\section{Hyperbolic automorphisms}
\label{convergence}

Let $\phi \in \Out(\FN)$ be a hyperbolic automorphism 
which possesses a properly expanding train track representative $f: \Gamma \to \Gamma$ as in Convention \ref{before-section-6}.
Let $[w]$ be a conjugacy class in $F_{N}$. Represent $[w]$ by a reduced loop 
$\gamma$ in $\Gamma$. Then the goodness of $w$, denoted by $\g([w])$, is defined by: \[
\g([w]):=\g(\gamma)\]

\begin{lem}
\label{rationalconvergence} 
Given a neighborhood $U$ of the simplex of attraction $\Delta_{+}(\varphi)\in\mathbb{P}\Curr(F_N)$, there exist a bound 
$\delta>0$ and an integer $M=M(U) \geq 1$ 
such that, for any 
$[w]\in F_N$ with $\g([w])\ge\delta$, we have \[
(\varphi^{M})^{n}[\eta_w]\in U
\]
for all $n\ge 1$.
\end{lem}
\setcounter{equation}{0}

\begin{proof} 
We first replace $\phi$ 
by a 
positive 
power as in Proposition \ref{goodness-growth} so that the goodness function for the train-track
map 
$f:\Gamma\to\Gamma$ 
becomes 
monotone. 

Recall from Section \ref{sec:currents}
that $[\nu]\in \mathbb{P}\Curr(F_N)$ is close to $[\nu']\in\mathbb{P}\Curr(F_N)$ if there exists  $\epsilon>0$ and $R>>0$ 
such that for all reduced edge paths 
$\gamma$ with $|\gamma|\le R$ 
we have \[
\left|\dfrac{\left<\gamma,\nu\right>}{\|\nu\|_{\Gamma}}-\dfrac{\left<\gamma,\nu'\right>}{\|\nu'\|_{\Gamma}}\right|<\epsilon.
\]
Thus, since $\Delta_+(\phi)$ is compact, there exist $\epsilon > 0$ and $R \in \R$ such that the above inequalities imply for $\nu' \in \Delta_+(\phi)$ that $\nu \in U$.

We proved the pointwise convergence for edges in Proposition \ref{edgeconvergence}. Since there are only finitely many edges and finitely many edge paths $\gamma$ in $\Gamma$ with $|\gamma|\le R$, we can pick an integer $M_{0}\ge0$ such that 
\begin{equation} \label{eq:edgepointwise}
\left|\dfrac{\left<\gamma,f^{n}(e)\right>}{|f^{n}(e)|}-\left<\gamma,\mu_{+}(e)\right>\right|<\epsilon/4
\end{equation}
for all $n\ge M_0$, for all $\gamma$ with $|\gamma|\le R$ and for all edges $e$ of $\Gamma$.  

Let $\lambda',\lambda''>1$ be the minimal and the maximal expansion factors respectively as given
in Convention \ref{tt-map-convention}. 
For any reduced loop $c$ in $\Gamma$ with
$\g(c) 
\geq \frac{1}{1+\epsilon/4}$ 
iterative application of the fact, that $f$ maps any good edge in $c$ to to a path in $[f(c)]$ which consists entirely of good edges and has length at least $\lambda'$, implies:
\[
|f^{n}(c)|\ge\frac{1}{1+\epsilon/4}|c|(\lambda')^{n}.
\]

Thus for any integer $M_{1}>\log_{\lambda'}{R(1+\frac{4}{\epsilon})}$ 
and 
all $n\ge M_1$ 
we get 
the following inequalities:
\begin{equation}
\label{eq:goodconjugacy}
\dfrac{R|c|}{|f^{n}(c)|
}
\le \dfrac{R|c|_\Gamma(1+\epsilon/4)}{|c|(\lambda')^{n}}\le\dfrac{R(1+\epsilon/4)}{(\lambda')^{M_1}}\le\dfrac{R(1+\epsilon/4)}{R(1+4/\epsilon)}
=
\frac{\epsilon}{4}
\end{equation}

We note that 
for any integer $m \geq 1$ and
for each edge the minimum expansion factor for $f^{m}$ is at least $(\lambda')^{m}$ and the maximum expansion factor for $f^{m}$ is at most $(\lambda'')^{m}$. 

For the rest of the proof set
\[
M=\max\{M_0,M_1\}
\]
and
\[
\delta :=\max\{\frac{1}{1+\epsilon/4},\frac{1}{1+(\frac{\lambda'}{\lambda''})^M\epsilon/4}\},
\]
and let $c$ be a reduced loop in $\Gamma$ which represents a conjugacy class $w$ with $\g(w)\ge\delta$. 

The assertion of Lemma \ref{rationalconvergence} now follows if we show
that for all integers $n \geq 1$ the current 
$(\varphi^{M})^n([\eta_w])$ is $(\epsilon, R)$-close (in the above sense) to some point in $\Delta_{+}(\varphi)$.  
Indeed, since by the first paragraph of the proof the goodness function is monotone, it suffices to assume $n = 1$ and apply the resulting statement iteratively.

For simplicity we 
denote from now on $f^M$ by $f$. Another auxiliary computation gives:
\begin{equation} \label{eq:badovergood}
\frac{(\lambda'')^M \cdot
\#\{ {\rm bad\ edges\ in\ c}\}}
{(\lambda')^M \cdot \#\{ {\rm good\ edges\ in\ c}\}}
=\dfrac{(\lambda'')^M(1-\g(w))|c|}{(\lambda')^M\g(w)|c|}
=(\dfrac{\lambda''}{\lambda'})^M\left(\dfrac{1}{\g(w)}-1\right)\le(\dfrac{\lambda''}{\lambda'})^M\left(\dfrac{1}{\delta}-1\right)\le
\frac{\epsilon}{4}.
\end{equation}

We now write 
$c=c_{1}c_{2}\dotsc b_{1}\dotsc c_{p}c_{p+1}\dotsc c_{r}b_{2}c_{r+1}\dotsc b_{k}
\ldots c_{s-1} 
c_{s}$, where the $c_{i}$ denote good edges and the $b_{j}$ denote maximal bad subpaths of $c$. 
Note for the second of the inequalities below that the definition of ``good'' implies that there can be no cancellation in $f(c)$ between adjacent $f(c_i)$ and $f(c_{i+1})$ nor between adjacent $f(c_i)$ and $f(b_{j})$, see Lemma \ref{growth-of-good}.

Then we 
calculate:
\begin{align*}
&\left|\dfrac{\left<\gamma,f(c)\right>}{|f(c)|}-\dfrac{\left<\gamma,|f(c_1)|\mu_{+}(c_1)+\dotsc+|f(c_s)|\mu_{+}(c_s)\right>}{\sum_{i=1}^{s}|f(c_i)|}\right|
\\
&
 \le \left|\dfrac{\left<\gamma,f(c)\right>}{|f(c)|}-\sum_{i=1}^{s}\dfrac{\left<\gamma,f(c_{i})\right>}{|f(c)|}\right|
\\
&+\left|\sum_{i=1}^{s}\dfrac{\left<\gamma,f(c_{i})\right>}{|f(c)|}-\dfrac{\sum_{i=1}^{s}\left<\gamma,f(c_i)\right>}{\sum_{i=1}^{s}|f(c_i)|}\right|
\\
&+\left|\dfrac{\sum_{i=1}^{s}\left<\gamma,f(c_i)\right>}{\sum_{i=1}^{s}|f(c_i)|}-\dfrac{\left<\gamma,|f(c_1)|\mu_{+}(c_1)+\dotsc+|f(c_s)|\mu_{+}(c_s)\right>}{\sum_{i=1}^{s}|f(c_i)|}\right|
\\
& \le \dfrac{2R|c|}{|f(c)|}+\sum_{j=1}^{k}\dfrac{\left<\gamma,[f(b_{j})]\right>}{|f(c)|}
\\
&+\left|\dfrac{\sum_{i=1}^{s}\left<\gamma,f(c_i)\right>}{\sum_{i=1}^{s}|f(c_i)|+\sum_{j=1}^{k}|[f(b_j)]|}-\dfrac{\sum_{i=1}^{s}\left<\gamma,f(c_i)\right>}{\sum_{i=1}^{s}|f(c_i)|}\right|
\\
&+\left|\dfrac{\sum_{i=1}^{s}\left<\gamma,f(c_i)\right>}{\sum_{i=1}^{s}|f(c_i)|}-\dfrac{\left<\gamma,|f(c_1)|\mu_{+}(c_1)+\dotsc+|f(c_s)|\mu_{+}(c_s)\right>}{\sum_{i=1}^{s}|f(c_i)|}\right|
\\
&< \epsilon/4+\epsilon/4+\epsilon/4+\epsilon/4=\epsilon.
\end{align*}

Here the first inequality is just a triangle inequality.  In the second inequality the last two terms are unchanged, and first two terms come from the first term in the previous quantity and follows from counting frequencies as follows: An occurrence of an edge path $\gamma$ or its inverse 
can occur either inside the image of a good edge $f(c_i)$ or inside of the image of a bad segment $[f(b_j)]$, or it might cross over the concatenation points. This observation gives the claimed inequality. In the final inequality, 
the first $\epsilon/4$ 
follows from the equation (\ref{eq:goodconjugacy}). The second one follows from the equation (\ref{eq:badovergood}) as follows: 
\[
\sum_{j=1}^{k}\dfrac{\left<\gamma,[f(b_{j})]\right>}{|f(c)|}\le\sum_{j=1}^{k}\dfrac{|[f(b_{j})]|}{|f(c)|}\le\sum_{j=1}^{k}\dfrac{|b_{j}|(\lambda'')^M}{(\lambda')^M
\ss
\cdot \#\{\textit{\rm good edges in $c$}\}
}\le \frac{\epsilon}{4}.
\]

The 
third $\epsilon/4$ comes from the following observation: 
\begin{align*}
&\left|\dfrac{\sum_{i=1}^{s}\left<\gamma,f(c_i)\right>}{\sum_{i=1}^{s}|f(c_i)|+\sum_{j=1}^{k}|[f(b_j)]|}-\dfrac{\sum_{i=1}^{s}\left<\gamma,f(c_i)\right>}{\sum_{i=1}^{s}|f(c_i)|}\right|\\
&=\left|\dfrac{\big(\sum_{i=1}^{s}\left<\gamma,f(c_i)\right>\big)\big(\sum_{j=1}^{k}|[f(b_j)]|\big)}{\big(\sum_{i=1}^{s}|f(c_i)|\big)\big(\sum_{i=1}^{s}|f(c_i)|+\sum_{j=1}^{k}|[f(b_j)]|\big)}\right|\\
&
\leq
\left|\dfrac{\sum_{i=1}^{s}\left<\gamma,f(c_i)\right>
}{
\sum_{i=1}^{s}|f(c_i)|}
\cdot
\dfrac{\sum_{j=1}^{k}|[f(b_j)]|}{
\sum_{i=1}^{s}|f(c_i)|}\right|\\
&\le\dfrac{(\lambda'')^M\sum_{j=1}^{k}|b_j|}{(\lambda')^M\sum_{i=1}^{s}|c_i|}\le\epsilon/4
\end{align*}
by (\ref{eq:badovergood}). 

Finally, the last $\epsilon/4$ can be verified using (\ref{eq:edgepointwise}) as follows:

\begin{align*}
&\left|\dfrac{\sum_{i=1}^{s}\left<\gamma,f(c_i)\right>}{\sum_{i=1}^{s}|f(c_i)|}-\dfrac{\left<\gamma,|f(c_1)|\mu_{+}(c_1)+\dotsc+|f(c_s)|\mu_{+}(c_s)\right>}{\sum_{i=1}^{s}|f(c_i)|}\right|\\
&=\left|\dfrac{\sum_{i=1}^{s}|f(c_i)|(\dfrac{\left<\gamma,f(c_i)\right>}{|f(c_i)|}-\left<\gamma,\mu_{+}(c_i)\right>)}{\sum_{i=1}^{s}|f(c_i)|}\right|\\
&\le \dfrac{\sum_{i=1}^{s}|f(c_i)|\epsilon/4}{\sum_{i=1}^{s}|f(c_i)|}=\epsilon/4
\end{align*}

Since after applying $\varphi^{M}$ to a conjugacy class with $\g(w)\ge\delta$ we still have $\g(\varphi^{M}(w))\ge\delta$ we get $(\varphi^{M})^{n}([\eta_w])\in U$ for all $n\ge1$. 
\end{proof}

\begin{lem} \label{convgood} 
For 
any $\delta>0$ and any 
neighborhood $U$ of $\Delta_{+}
(\phi)$ there exists an integer $M(\delta, U)>0$ such that (up to replacing $\varphi$ by a positive
power) for any conjugacy class
$w$ with goodness $\g(w)>\delta$ we have
$$\varphi^n[\eta_w]\in U$$
for all $n\ge M$. 
\end{lem}

\begin{proof} This is a direct consequence of Corollary 
\ref{goodbig} and Lemma \ref{rationalconvergence}. 
\end{proof}

\begin{prop}\label{backforth}
Let $\varphi\in Out(F_N)$ be a hyperbolic outer automorphism such that $\varphi$ and $\varphi^{-1}$ both admit properly expanding train-track representatives.

Given neighborhoods $U$ of the simplex of attraction $\Delta_{+}
(\phi)$ and $V$ of the simplex of repulsion $\Delta_-
(\phi)$ in $\mathbb{P}\Curr(F_N)$, 
then (up to replacing $\varphi$ by a positive 
power) there exist an integer $M\ge 0$ such that for any conjugacy class $w\in F_N$ we have
\[
\varphi^{n}[\eta_w]\in U \qquad {\rm or} \qquad  \varphi^{-n}[\eta_w]\in V
\]
for all $n\ge M$. 
\end{prop}

\begin{proof} This is a direct consequence of Proposition \ref{backforthgoodness} 
(see Remark \ref{back-to-ht}) 
and Lemma \ref{convgood}. 
\\${}^{}$
\end{proof}

The following result also proves Theorem \ref{mainthm-intro} from the Introduction:

\begin{thm}
\label{mainthm} 
Let $\varphi\in Out(F_N)$ be a hyperbolic outer automorphism such that $\varphi$ and $\varphi^{-1}$ admit absolute train-track representatives. Then, $\varphi$ acts on $\mathbb{P}\Curr(F_N)$ with uniform North-South 
dynamics from $\Delta_-(\phi)$ to $\Delta_+(\phi)$:

Given an open neighborhood $U$ of the  
simplex of attraction $\Delta_{+}(\varphi)$ and a compact set $K\subset\mathbb{P}\Curr(F_N)\smallsetminus\Delta_{-}(\varphi)$ there exists an integer $M>0$ such that $\varphi^{n}(K)\subset U$ for all $n\ge M$. 
\end{thm}

\begin{proof} 
According to Remark \ref{proper-tts-exist} we can 
pass to common positive powers of $\phi$ and $\phi^{-1}$ that have a properly expanding train track representatives as in Convention \ref{before-section-6}. We can then combine
Propositions 
\ref{current-space}, 
\ref{convergence-criterion}, \ref{NS-for-roots} and \ref{backforth} to obtain 
directly 
the required result. 
For the application of Proposition \ref{convergence-criterion} we need that 
$\Delta_+(\phi)$ and $\Delta_-(\phi)$
are disjoint, which is shown in 
Remark \ref{dynamics-within} below.
We also need for Proposition \ref{convergence-criterion} that both, $\Delta_+(\phi)$ and $\Delta_-(\phi)$, are $\phi$-invariant, which is shown in Corollary \ref{delta-invariance}.
\end{proof}

\begin{rem}[\bf{Dynamics within $\Delta_{+}(\phi)$}] 
\label{dynamics-within}
(1)
It is proved in Proposition \ref{edgeconvergence} that every vertex of the 
simplex 
$\Delta_+ = \Delta_+(\phi)$ is an {\em expanding $\phi$-invariant} current,
i.e. a projectivized current $[\mu]$ for which there exist $\lambda > 1$ and $t \geq 1$ such that $\phi^t(\mu) = \lambda \mu$. 

For the rest of this remark we replace $\phi$ by a suitable positive power so that every vertex current of $\Delta_+$ is projectively fixed by $\phi$. Of course, this implies that $\phi$ fixes also every face $\Delta'$ of $\Delta_+$ (but not necessarily pointwise).

\smallskip
\noindent
(2)
A \emph{uniform face} 
$\Delta'$ 
is a face of $\Delta_+$ which is spanned by 
vertices 
$[\mu^1_{+}],\ldots, [\mu^{k}_{+}]$ 
that 
all 
have the same stretch factor
$\lambda > 1$, i.e. $\varphi(\mu^j_{+})=\lambda\mu^j_{+}$ 
for all $1\le j\le k$. 

\smallskip
\noindent
(3)
Since the action of $\varphi$ on $\Curr(F_N)$ is linear, all the uniform faces 
of $\Delta_+$ 
are pointwise fixed. 
Any non-vertex  
current $[\mu]$ is always 
contained in 
the interior of 
some 
face $\Delta' \subset \Delta_{+}$. 
Then 
the sequence of 
$\varphi^{n}([\mu])$ converges towards 
a point in the uniform face 
$\Delta'_+$
of $\Delta'$ which is spanned by all vertices that have maximal stretch factor among all vertices of $\Delta'$.

\smallskip
\noindent
(4)
Similarly, under backwards iteration $n \to -\infty$
the sequence of $\varphi^{n}([\mu])$ converges towards 
a point in the uniform face $\Delta'_-$ of 
$\Delta' \subset \Delta_+$
which is spanned by all vertices that have minimal stretch factor among all vertices of $\Delta'$.

\smallskip
\noindent
(5)
As a consequence, it follows that a current $[\mu]$ can not belong to both, $\Delta_+(\phi)$ and $\Delta_-(\phi)$: By symmetry in the previous paragraph, any $[\mu] \in \Delta_-(\phi)$ converges under backwards iteration of $\phi^{-1}$ to an expanding $\phi^{-1}$-invariant current. But backwards iteration of $\phi^{-1}$ is the same as forward iteration of $\phi$, and the limit current can not be simultaneously expanding $\phi$-invariant and expanding $\phi^{-1}$-invariant.
\end{rem}

\begin{rem}
\label{BHL}
As recalled in Remark \ref{dynamics-within}, every vertex of the convex cell $\Delta_+ = \Delta_+(\phi)$ is an expanding $\phi$-invariant current. Furthermore, any point in a uniform face of $\Delta_+$ is also an expanding $\phi$-invariant current, and these are the only such 
in $\Delta_+$. It follows from Theorem \ref{mainthm} and from Remark \ref{dynamics-within} (applied to $\phi^{-1}$) that there are no other expanding $\phi$-invariant currents in $\PCurr \smallsetminus \Delta_+$.

It has been shown in 
\cite{BHL} in a slightly more general context that {\em expanding fixed currents} (i.e. currents $\mu$ with $\phi\mu = \lambda \phi$ for some $\lambda > 1$)
are in 1-1 relation with the non-negative eigenvectors with eigenvalue $> 1$ for the transition matrix $M(f)$ of any expanding train track representative $f: \Gamma \to \Gamma$ of $\phi$.
It seems likely that a similar result can be obtained by extending the methods of Section \ref{symbolic}.

\end{rem}

Moreover, pointwise dynamics for a dense subset follows from our machinery: 

\begin{thm} 
\label{thm-rational-limits}
Let $\eta_g$ be a rational current in $\Curr(F_N)$. Then, there exist  $[\mu_{\infty}]\in \Delta_{+}$ such that 
\[
\lim_{t\to\infty}\varphi^{t}[\eta_g]=[\mu_{\infty}].\]
\end{thm}

\begin{proof} Let $\gamma$ be a reduced loop 
representing the conjugacy class $g$. Apply a 
sufficiently large 
power $\varphi^{k}$ to $g$, hence 
$[f^{k}(\gamma)]$, the reduced edge path 
representing $\varphi^{k}(g)$ is 
a pseudo-legal loop, see 
Proposition 
\ref{pseudo-legal-iterate}. Then the arguments in 
in the proof
of Lemma \ref{rationalconvergence} give the pointwise convergence by looking at the set of 
non-INP edges in 
$[f^{k}(\gamma)]$.
\end{proof}

\begin{rem}
\label{rational-limit-current}
From the proof of Theorem \ref{thm-rational-limits} 
and from the results of our previous paper \cite{LU1} 
one 
can also derive 
a precise description of the limit current $\mu_\infty$ for any given rational current $\eta_g$: Let $[f^{k}(\gamma)]$ be the pseudo-legal loop 
representing the conjugacy class $\varphi^{k}(g)$, and assume without loss of generality that the initial point doesn't lie on any INP-subpath. Let 
$$[f^{k}(\gamma)] = c_1 c_1 \ldots c_p \rho_1 c_{p+1} c_{p+1} \ldots \rho_2 c_{r} c_{r+1} \ldots c_{s-1} c_s$$
be a legal concatenation of edges $c_i$ and of INP's $\rho_j$. Then the limit current $\mu_\infty$ is given by
$$[\mu_\infty] = [\sum_{c_i \in E\Gamma'} \mu_+(c_i)]\, ,$$
where $\Gamma'$ is the (possibly non-connected) subgraph of $\Gamma$ consisting precisely of those strata $H_i$ of $\Gamma$ 
which satisfy one of the following two conditions, where $\succeq$ denotes the natural partial order on strata defined by the non-zero off-diagonal coefficients of the transition matrix $M(f)$ (see \cite{LU1}):
\begin{enumerate}
\item
There exists a chain $H_i = H_d \succeq H_{d-1} \succeq \ldots \succeq H_1$ of pairwise distinct strata, all with maximal Perron-Frobenius eigenvalue among the diagonal blocks of $M(f)$. Furthermore $d$ is the maximal length of any such chain.
\item
$H_i \succeq H_j$ for some $H_j$ as in (1).
\end{enumerate}
\end{rem}

The analogue of 
Theorem \ref{thm-rational-limits} 
for non-rational currents is delicate, but the authors expect that it is actually true. However, as already cautioned in the Introduction, in general the pointwise convergence 
on $\PCurr$ 
will {\em not} be uniform.

\medskip

To finish the paper, we would like to illustrate our main result by considering 
the particular case of a hyperbolic $\phi \in \Out(\FN)$ which is given as
a ``free product'' of 
fully irreducible hyperbolic  
outer automorphisms: 
Theorem \ref{mainthm}
then shows that such a free product automorphism acts on the space of projectivized geodesic currents with generalized uniform North-South dynamics, and the simplex of attraction is spanned by 
the forward limit currents
for each of the fully irreducible factors.  
More precisely: 

\begin{exmp}
\label{fee-products} 
Let $\varphi_i\in Out(F_{N_i})$ 
(with 
$N_i\ge3$) be a finite family of 
hyperbolic fully irreducible 
outer automorphisms, where 
each $\varphi_{i}:F_{N_i}\to F_{N_i}$ is represented by an 
absolute 
train-track map $f_i:\Gamma_i\to\Gamma_i$.

For simplicity we now assume (for example by passing to a positive power) that every $f_i$ has a fixed vertex $v_i$ in $\Gamma_i$, and we consider the lift $\Phi_i \in \Aut(\FN)$ represented by $f_i$ on $\pi_1(\Gamma_i, v_i)$. 
We then get a global 
absolute 
train-track map $f$ 
by wedging 
together the 
$\Gamma_i$ 
at 
the fixed points $v_i$, and $f$ represents the outer automorphism $\phi$ given by $\Phi=\Phi_{1}*\ldots*\Phi_{k} \in \Aut(\FN)$, 
for $F_N :=F_{N_1}*\ldots *F_{N_{k}}$.

Then the 
Bestvina-Feighn Combination Theorem implies that the ``free product'' $\phi$
is hyperbolic. 
We 
consider for 
any 
$i$ the embedding 
$\kappa_i: \Curr(F_{N_i})\hookrightarrow \Curr(F_N)$ as described in \cite{KL5}. 

In \cite{U2} it was proven that 
for each factor graph $\Gamma_i$ 
the current $\mu_{+}(e)$ defined in Proposition \ref{edgeconvergence} is the same for every edge $e$ of $\Gamma_i$. Hence the simplex of attraction for $\varphi$ is spanned by $k$ currents $\mu_+(e_i)$, each of which 
is given by an edge $e_i$ 
from a different $\Gamma_i$. In fact, each 
$\mu_{+}(e_i)$ 
is the 
$\kappa_i$-image of 
the well defined forward limit current
for $\varphi_i$ in $\Curr(F_N)$. Analogously, the simplex of repulsion is equal to the linear span of the images 
 in $\Curr(F_N)$ 
 of 
the backwards limit 
currents for 
all 
$\varphi$.
\end{exmp}

\bibliographystyle{alpha}
\bibliography{dynamicsoncurrents}

\end{document}